\documentclass[bj,authoryear,noshowframe]{imsart}

\usepackage{amsfonts, amsthm, amsmath, amssymb, graphicx, tikz, tikz-3dplot, nicefrac}
\usetikzlibrary{matrix, calc, backgrounds, arrows.meta}

\startlocaldefs
\theoremstyle{plain}

\newtheorem{theorem}{Theorem}[section]
\newtheorem{lemma}[theorem]{Lemma}
\newtheorem{proposition}[theorem]{Proposition}
\theoremstyle{definition}
\newtheorem{definition}[theorem]{Definition}

\newtheorem{remark}{Remark}
\newcommand{\X}{\mathcal X}
\newcommand{\Y}{\mathcal Y}

\newcommand{\R}{\mathbb R}

\newcommand{\Z}{\mathbb Z}
\newcommand{\m}{\mathfrak m}
\newcommand{\dd}{\rho}
\newcommand{\ddd}{\mathrm{d}}
\newcommand{\vol}{\mathrm{vol}}
\newcommand{\diam}{\mathrm{diam}}
\newcommand{\supp}{\mathrm{supp}}

\newcommand{\PhiEx}{\Phi_{\mathrm{ex}}}

\DeclareMathOperator*{\argmin}{arg\,min}

\usetikzlibrary{calc,backgrounds}
\definecolor{nestcol}{RGB}{235,145,25}
\definecolor{balcol}{RGB}{63,150,63}
\definecolor{twocol}{RGB}{0,170,190}
\definecolor{threecol}{RGB}{198,77,160}
\newcommand{\glyphscale}{4.3mm}
\tikzset{
  graph edge/.style={draw=black!62, line width=0.65pt},
  graph vertex/.style={circle, fill=black, inner sep=0pt, minimum size=2.2pt},
  glyph edge/.style={draw=black, line width=0.33pt, line cap=round, line join=round},
  internal/.style={circle, fill=black, inner sep=0pt, minimum size=1.65pt},
  root/.style={circle, fill=black, inner sep=0pt, minimum size=1.95pt},
  leaf/.style={circle, draw=black, fill=white, inner sep=0pt, minimum size=1.65pt, line width=0.16pt},
  leaf number/.style={font=\tiny, inner sep=0pt},
  ray2box/.style={draw=twocol, rounded corners=1.2pt, fill=white, inner sep=0.9pt, line width=0.55pt},
  ray3box/.style={draw=threecol, rounded corners=1.2pt, fill=white, inner sep=0.9pt, line width=0.55pt},
  nestbox/.style={draw=nestcol, rounded corners=1.1pt, fill=white, inner sep=0.8pt, line width=0.50pt},
  balbox/.style={draw=balcol, rounded corners=1.1pt, fill=white, inner sep=0.8pt, line width=0.50pt},
  apexbox/.style={draw=black!65, rounded corners=1.2pt, fill=gray!7, inner sep=0.9pt, line width=0.55pt},
  legendbox/.style={draw=black!20, rounded corners=1.6pt, fill=white, inner sep=2.0pt},
  every node/.style={font=\scriptsize}
}
\newcommand{\Leaf}[3]{
  \node[leaf] at (#2,#3) {};
  \node[leaf number] at ($(#2,#3)+(0,-0.19)$) {#1};
}
\newcommand{\TreeTwo}[4]{
\begin{tikzpicture}[x=\glyphscale,y=\glyphscale,baseline=-0.25ex]
  \draw[glyph edge] (0,1.18)--(-.62,.06);
  \draw[glyph edge] (0,1.18)--(0,.55);
  \draw[glyph edge] (0,1.18)--(.62,.06);
  \draw[glyph edge] (0,.55)--(-.20,0);
  \draw[glyph edge] (0,.55)--(.20,0);
  \node[root] at (0,1.18) {};
  \node[internal] at (0,.55) {};
  \Leaf{#3}{-.62}{.06}
  \Leaf{#1}{-.20}{0}
  \Leaf{#2}{ .20}{0}
  \Leaf{#4}{ .62}{.06}
\end{tikzpicture}
}
\newcommand{\TreeThree}[4]{
\begin{tikzpicture}[x=\glyphscale,y=\glyphscale,baseline=-0.25ex]
  \draw[glyph edge] (0,1.18)--(-.64,.04);
  \draw[glyph edge] (0,1.18)--(.05,.55);
  \draw[glyph edge] (.05,.55)--(-.26,0);
  \draw[glyph edge] (.05,.55)--(.05,0);
  \draw[glyph edge] (.05,.55)--(.36,0);
  \node[root] at (0,1.18) {};
  \node[internal] at (.05,.55) {};
  \Leaf{#4}{-.64}{.04}
  \Leaf{#1}{-.26}{0}
  \Leaf{#2}{ .05}{0}
  \Leaf{#3}{ .36}{0}
\end{tikzpicture}
}
\newcommand{\TreeNested}[4]{
\begin{tikzpicture}[x=\glyphscale,y=\glyphscale,baseline=-0.25ex]
  \draw[glyph edge] (0,1.25)--(-.66,.07);
  \draw[glyph edge] (0,1.25)--(0,.82);
  \draw[glyph edge] (0,.82)--(-.22,.42);
  \draw[glyph edge] (0,.82)--(.40,.06);
  \draw[glyph edge] (-.22,.42)--(-.42,0);
  \draw[glyph edge] (-.22,.42)--(-.05,0);
  \node[root] at (0,1.25) {};
  \node[internal] at (0,.82) {};
  \node[internal] at (-.22,.42) {};
  \Leaf{#4}{-.66}{.07}
  \Leaf{#1}{-.42}{0}
  \Leaf{#2}{-.05}{0}
  \Leaf{#3}{ .40}{.06}
\end{tikzpicture}
}
\newcommand{\TreeBalanced}[4]{
\begin{tikzpicture}[x=\glyphscale,y=\glyphscale,baseline=-0.25ex]
  \draw[glyph edge] (0,1.16)--(-.34,.54);
  \draw[glyph edge] (0,1.16)--(.34,.54);
  \draw[glyph edge] (-.34,.54)--(-.55,0);
  \draw[glyph edge] (-.34,.54)--(-.18,0);
  \draw[glyph edge] (.34,.54)--(.18,0);
  \draw[glyph edge] (.34,.54)--(.55,0);
  \node[root] at (0,1.16) {};
  \node[internal] at (-.34,.54) {};
  \node[internal] at (.34,.54) {};
  \Leaf{#1}{-.55}{0}
  \Leaf{#2}{-.18}{0}
  \Leaf{#3}{ .18}{0}
  \Leaf{#4}{ .55}{0}
\end{tikzpicture}
}
\newcommand{\TreeStar}{
\begin{tikzpicture}[x=\glyphscale,y=\glyphscale,baseline=-0.25ex]
  \draw[glyph edge] (0,1.15)--(0,.60);
  \draw[glyph edge] (0,.60)--(-.58,0);
  \draw[glyph edge] (0,.60)--(-.20,0);
  \draw[glyph edge] (0,.60)--(.20,0);
  \draw[glyph edge] (0,.60)--(.58,0);
  \node[root] at (0,1.15) {};
  \node[internal] at (0,.60) {};
  \Leaf{1}{-.58}{0}
  \Leaf{2}{-.20}{0}
  \Leaf{3}{ .20}{0}
  \Leaf{4}{ .58}{0}
\end{tikzpicture}
}
\newcommand{\tikzsphsetvalued}{
\begin{tikzpicture}[scale=1.0]
\begin{scope}[shift={(0,0)}]
  \draw[thick] (0,0) circle (2);
  \draw[thick] (-2,0) arc[start angle=180,end angle=360,x radius=2,y radius=0.55];
  \draw[densely dashed] (-2,0) arc[start angle=180,end angle=0,x radius=2,y radius=0.55];
  \fill[red] (0,2) circle (3.5pt) node[above,black] {$N$};
  \fill[red] (0,-2) circle (3.5pt) node[below,black] {$S$};
  \draw[very thick,blue] (-2,0) arc[start angle=180,end angle=360,x radius=2,y radius=0.55];
  \draw[very thick,blue,densely dashed] (-2,0) arc[start angle=180,end angle=0,x radius=2,y radius=0.55];
  \node[align=center,font=\scriptsize] at (0,-2.85) {$\frac12\delta_N+\frac12\delta_S$\\mean set = equator};
\end{scope}
\begin{scope}[shift={(5.7,0)}]
  \draw[thick] (0,0) circle (2);
  \draw[thick] (-2,0) arc[start angle=180,end angle=360,x radius=2,y radius=0.55];
  \draw[densely dashed] (-2,0) arc[start angle=180,end angle=0,x radius=2,y radius=0.55];
  \fill[blue] (0,2) circle (3.5pt) node[above,black] {$N$};
  \fill[blue] (0,-2) circle (3.5pt) node[below,black] {$S$};
  \draw[very thick,red] (-2,0) arc[start angle=180,end angle=360,x radius=2,y radius=0.55];
  \draw[very thick,red,densely dashed] (-2,0) arc[start angle=180,end angle=0,x radius=2,y radius=0.55];
  \node[align=center,font=\scriptsize] at (0,-2.85) {uniform on equator\\mean set = $\{N,S\}$};
\end{scope}
\end{tikzpicture}
}
\newcommand{\CherryTree}[5]{%
  \begin{scope}[shift={#1}, scale=0.3]
    \draw[rounded corners=2pt, draw=#2, line width=0.6pt, fill=white]
      (-0.72,-0.10) rectangle (0.72,1.22);

    \coordinate (R) at (0,1.00);     
    \coordinate (C) at (-0.12,0.60); 
    \coordinate (L1) at (-0.38,0.22);
    \coordinate (L2) at ( 0.00,0.22);
    \coordinate (L3) at ( 0.38,0.22);

    \draw[gray!70, line width=0.6pt] (R)--(C);
    \draw[gray!70, line width=0.6pt] (R)--(L3);
    \draw[gray!70, line width=0.6pt] (C)--(L1);
    \draw[gray!70, line width=0.6pt] (C)--(L2);

    \fill (R) circle (1.4pt);
    \fill (C) circle (1.2pt);

    \draw[draw=gray!75, fill=white, line width=0.5pt] (L1) circle (1.5pt);
    \draw[draw=gray!75, fill=white, line width=0.5pt] (L2) circle (1.5pt);
    \draw[draw=gray!75, fill=white, line width=0.5pt] (L3) circle (1.5pt);

    \node[font=\scriptsize] at ($(L1)+(0,-0.10)$) {#3};
    \node[font=\scriptsize] at ($(L2)+(0,-0.10)$) {#4};
    \node[font=\scriptsize] at ($(L3)+(0,-0.10)$) {#5};
  \end{scope}
}

\newcommand{\tikzputhree}{
\begin{tikzpicture}[scale=1.8, line cap=round, line join=round]

  \tikzset{
    simplex/.style={
      draw=gray!90!black,
      very thick,
      fill=gray!10
    },
    metric/.style={
      draw=blue!50!black,
      very thick,
      fill=blue!25
    },
    tripod/.style={
      draw=red!80!black,
      very thick
    },
    pointblue/.style={
      circle,
      fill=blue!70!black,
      inner sep=2pt
    },
    pointred/.style={
      circle,
      fill=red!80!black,
      inner sep=2pt
    },
    pointblack/.style={
      circle,
      fill=black,
      inner sep=2pt
    }
  }

  \coordinate (A) at (0,0);
  \coordinate (B) at (6,0);
  \coordinate (C) at (3,{3*sqrt(3)});

  \coordinate (M23) at (3,0);                    
  \coordinate (M13) at (1.5,{1.5*sqrt(3)});      
  \coordinate (M12) at (4.5,{1.5*sqrt(3)});      

  \coordinate (G) at (3,{sqrt(3)});              

  \draw[simplex] (A)--(B)--(C)--cycle;

  \draw[metric] (M12)--(M13)--(M23)--cycle;

  \draw[tripod] (G)--(M12);
  \draw[tripod] (G)--(M13);
  \draw[tripod] (G)--(M23);

  \node[pointblack] at (A) {};
  \node[pointblack] at (B) {};
  \node[pointblack] at (C) {};

  \node[below left]  at (A) {$(1,0,0)$};
  \node[below right] at (B) {$(0,1,0)$};
  \node[above]       at (C) {$(0,0,1)$};

  \node[pointred] at (M12) {};
  \node[pointred] at (M13) {};
  \node[pointred] at (M23) {};

    \CherryTree{($(3.75,2.165063509461-0.1)$)}{red!70!black}{1}{2}{3}
    \CherryTree{($(2.25,2.165063509461-0.1)$)}{red!70!black}{1}{3}{2}
    \CherryTree{($(3,0.866025403784)$)}{red!70!black}{2}{3}{1}

  \node[right] at (M12)
    {$\left(0,\nicefrac12,\nicefrac12\right)$};

  \node[left] at (M13)
    {$\left(\nicefrac12,0,\nicefrac12\right)$};

  \node[below] at (M23)
    {$\left(\nicefrac12,\nicefrac12,0\right)$};

  \node[pointred] at (G) {};
  \node[below] at (G)
    {$\left(\nicefrac13,\nicefrac13,\nicefrac13\right)$};

  \node[gray!70!black, font=\large] at (3,4.45)
    {$\Delta_3$};

  \node[blue!70!black, font=\large] at (3,2.9)
    {$\widetilde{\mathcal M}_3$};

  \node[red!80!black, font=\large] at (3,2.1)
    {$\widetilde{\mathcal T}_3$};

\end{tikzpicture}
}
\newcommand{\tikzpufourpetersen}{
\begin{tikzpicture}[scale=1.0]
\def\R{7.35}
\def\r{3.25}
\coordinate (V14)  at ( 90:\R);
\coordinate (V124)  at ( 18:\R);
\coordinate (V12) at (-54:\R);
\coordinate (V34)  at (-126:\R);
\coordinate (V134) at ( 162:\R);
\coordinate (W23) at ( 90:\r);
\coordinate (W24) at ( 18:\r);
\coordinate (W123)  at (-54:\r);
\coordinate (W234)  at (-126:\r);
\coordinate (W13)  at ( 162:\r);
\coordinate (Apex) at (0,0);
\begin{scope}[on background layer]
  \draw[graph edge] (V12)--(V34)--(V134)--(V14)--(V124)--cycle;
  \draw[graph edge] (W123)--(W13)--(W24)--(W234)--(W23)--cycle;
  \draw[graph edge] (V12)--(W123);
  \draw[graph edge] (V34)--(W234);
  \draw[graph edge] (V134)--(W13);
  \draw[graph edge] (V14)--(W23);
  \draw[graph edge] (V124)--(W24);
\end{scope}
\foreach \P in {V12,V34,V134,V14,V124,W123,W234,W13,W23,W24}
  \node[graph vertex] at (\P) {};
\node[ray2box] at ($(V12)+(0.1,-0.45)$)     {\TreeTwo{1}{2}{3}{4}};
\node[ray2box] at ($(V34)+(-0.1,-0.45)$)  {\TreeTwo{3}{4}{1}{2}};
\node[ray3box] at ($(V134)+(-0.42,0.16)$){\TreeThree{1}{3}{4}{2}};
\node[ray2box] at ($(V14)+(0,0.45)$)    {\TreeTwo{1}{4}{2}{3}};
\node[ray3box] at ($(V124)+(0.42,0.16)$){\TreeThree{1}{2}{4}{3}};
\node[ray3box] at ($(W123)+(0.2,-0.45)$)    {\TreeThree{1}{2}{3}{4}};
\node[ray3box] at ($(W234)+(-0.2,-0.45)$) {\TreeThree{2}{3}{4}{1}};
\node[ray2box] at ($(W13)+(-0.42,0.15)$) {\TreeTwo{1}{3}{2}{4}};
\node[ray2box] at ($(W23)+(0,0.45)$)    {\TreeTwo{2}{3}{1}{4}};
\node[ray2box] at ($(W24)+(0.42,0.15)$) {\TreeTwo{2}{4}{1}{3}};
\node[balbox]  at ($(V12)!0.50!(V34)+(0,0)$)  {\TreeBalanced{1}{2}{3}{4}}; 
\node[nestbox] at ($(V34)!0.50!(V134)+(0,0)$){\TreeNested{3}{4}{1}{2}};  
\node[nestbox] at ($(V134)!0.50!(V14)+(0,0)$){\TreeNested{1}{4}{3}{2}};  
\node[nestbox] at ($(V14)!0.50!(V124)+(0,0)$){\TreeNested{1}{4}{2}{3}}; 
\node[nestbox] at ($(V124)!0.50!(V12)+(0,0)$){\TreeNested{1}{2}{4}{3}}; 
\node[nestbox] at ($(W123)!0.50!(W13)+(0,0)$){\TreeNested{1}{3}{2}{4}}; 
\node[balbox]  at ($(W13)!0.50!(W24)+(0,0)$){\TreeBalanced{1}{3}{2}{4}};
\node[nestbox] at ($(W24)!0.50!(W234)+(0,0)$){\TreeNested{2}{4}{3}{1}}; 
\node[nestbox] at ($(W234)!0.50!(W23)+(0,0)$){\TreeNested{2}{3}{4}{1}}; 
\node[nestbox] at ($(W23)!0.50!(W123)+(0,0)$){\TreeNested{2}{3}{1}{4}}; 
\node[nestbox] at ($(V12)!0.50!(W123)+(0,0)$){\TreeNested{1}{2}{3}{4}}; 
\node[nestbox] at ($(V34)!0.50!(W234)+(0,0)$){\TreeNested{3}{4}{2}{1}}; 
\node[nestbox] at ($(V134)!0.50!(W13)+(0,0)$){\TreeNested{1}{3}{4}{2}}; 
\node[balbox]  at ($(V14)!0.50!(W23)+(0,0)$){\TreeBalanced{1}{4}{2}{3}};
\node[nestbox] at ($(V124)!0.50!(W24)+(0,0)$){\TreeNested{2}{4}{1}{3}}; 
\node[apexbox] at (Apex) {\TreeStar};
\end{tikzpicture}
}

\endlocaldefs

\begin{document}

\begin{frontmatter}
\title{Uniform Consistency of Generalized Fr\'echet Means}
\runtitle{Generalized Fr\'echet Means}

\begin{aug}
\author[C]{\inits{A.}\fnms{Andrea}~\snm{Aveni}}
\author[C]{\inits{M.}\fnms{Martin}~\snm{Bladt}}
\author[B,D]{\inits{S.}\fnms{Sayan}~\snm{Mukherjee}\thanksref{sayanmem}}

\address[C]{Department of Mathematical Sciences, University of Copenhagen, Copenhagen, Denmark}
\address[B]{Department of Statistical Science, Duke University, Durham, North Carolina, USA}
\address[D]{Max Planck Institute for Mathematics in the Sciences, Leipzig, Germany}
\thankstext[\textdagger]{sayanmem}{Sayan Mukherjee passed away on 31 March 2025. This article is dedicated to his memory.}
\end{aug}

\begin{abstract}
Loss-based notions of centre on nonlinear spaces range from the Fr\'echet mean and power means to the geometric median and, in a limiting sense, the Chebyshev centre. To use such summaries statistically, one first needs a law of large numbers that remains valid beyond smooth manifolds and beyond a fixed choice of loss. We study generalized Fr\'echet means on metric spaces with the Heine--Borel property, obtained by replacing squared distance with a convex loss under a mild exponential-growth condition. We prove existence and compactness of the population mean set, establish a sharp diameter bound, obtain almost-sure consistency of empirical $\phi$-means, and derive a uniform strong law over compact classes of losses. The analysis is driven by a deterministic argmin principle together with a Glivenko--Cantelli theorem for monotone classes. For isotropic densities on Riemannian symmetric spaces, we identify the population $\phi$-mean for every strictly increasing loss for which the objective is finite, including bounded robust losses. We also illustrate the framework on spheres and on the polyhedral space of ultrametric phylogenetic trees.
\end{abstract}


\end{frontmatter}
\section{Introduction}\label{SSIntro}
The Fr\'echet mean of a probability measure $\mu$ on a metric space $(\X,\dd)$ is the minimizer of the function $x\mapsto\int\dd(x,y)^2\mu(\ddd y)$. Introduced by \citet{frechet1948elements}, it generalizes the arithmetic mean to non-Euclidean settings and is now the central object of statistical inference on Riemannian manifolds, of statistical shape analysis \citep{Karcher77}, and of computational anatomy. Its asymptotic theory was initiated by \citet{Bhatta, Bhatta2}, who established strong consistency and central limit theorems on manifolds; the hypotheses have been progressively relaxed in subsequent work, which we recall in Section \ref{SSLittRev}.

The squared distance in the definition is a convenient mathematical convention rather than a fundamental one. Replacing $\dd^2$ by $\dd^p$ for $p\geq 1$ yields the $p$-means \citep{BhattaBook}, with $p=1$ recovering the geometric median; replacing it by a general convex function $\phi:\R_+\to\R_+$ gives the \emph{$\phi$-mean}
\begin{equation}\label{eq:phi-mean}
\m^\phi_\mu:=\argmin_{x\in\X}\int_\X \phi(\dd(x,y))\,\mu(\ddd y).
\end{equation}
The Fr\'echet mean, and the geometric median are recovered by $\phi(t)=t^2$ and $\phi(t)=t$; for $\phi(t)=t^p$ with $p\to\infty$, Proposition \ref{chebyshev} below shows $\m^\phi_\mu$ converges to the Chebyshev centre of $\supp\mu$, minimizing the largest distance to the support. The family thus interpolates between the most regular and the most robust notions of centre, in the spirit of Huber, Tukey, and Andrews losses in robust statistics on $\R^d$.

Existing work has either fixed $\phi$ or restricted to a one-parameter family $\{t^p\}_{p\geq 1}$ or close variants. The frameworks developed by \citet{Scho, Scho2} are more general, as they admit cost functions that do not factor through $\dd$, but check a quadruple-inequality hypothesis for each cost separately and do not address joint behaviour over function classes. We ask when the empirical $\phi$-mean is a strongly consistent estimator of its population analogue, uniformly over $\phi$ in a class $\Psi$ of admissible losses. This law-of-large-numbers question is fundamental on nonlinear and possibly singular sample spaces, where the empirical centre is often the first and most basic statistical summary available. A uniform statement justifies reporting the path $\phi\mapsto\m^\phi_n$ as a robust summary of the centre of $\mu$, supports data-driven selection of $\phi$, and, in the case when all $\phi$-means in $\Psi$ collapse to a single point (as on a symmetric space with an isotropic density), yields a robustness conclusion that is uniform in the loss. To the best of our knowledge, existing strong laws for generalized Fr\'echet means are pointwise in the loss, or restricted to low-dimensional families such as $p$-means; they do not provide uniform consistency over compact classes of admissible losses on general metric spaces.

We work on metric spaces with the Heine--Borel property, that is, spaces in which closed bounded sets are compact. By Hopf--Rinow, this broad setting includes every complete connected Riemannian manifold with its geodesic metric; it also includes the canonical manifold-valued data spaces (Stiefel and Grassmann manifolds, the cone of SPD matrices with the affine-invariant metric, Kendall shape spaces, hyperbolic spaces) and several singular spaces of statistical interest, in particular finite polyhedral complexes and the ultrametric spaces of rooted phylogenetic trees considered in Section \ref{SSTrees}. Our admissible losses form a class of convex bijections subject to a mild exponential-growth control:
$$\PhiEx:=\Bigl\{\phi:\R_+\to\R_+\text{ convex bijection with }\gamma_\phi<\infty\Bigr\},\qquad\gamma_\phi:=\sup_{x\geq 0}\frac{\phi(x+1)}{\phi(1)+\phi(x)}.$$
The class $\PhiEx$ contains every power $\phi(t)=t^p$ for $p\geq 1$ (with $\gamma_\phi=2^{p-1}$), every rescaled exponential $\phi(t)=(a^t-1)/(a-1)$ for $a>1$ (with $\gamma_\phi=a$), and is closed under sums, products, and positive scalar multiples. The growth condition $\gamma_\phi<\infty$ is sharp: Lemma \ref{lemFinite} shows that on an unbounded complete locally compact length space, $\PhiEx$ is exactly the class for which the loss $F^\phi_\mu(x):=\int\phi(\dd(x,y))\,\mu(\ddd y)$ obeys the dichotomy ``finite everywhere or finite nowhere''. Outside $\PhiEx$, one can construct measures for which $F^\phi_\mu$ is finite on a non-trivial proper subset of $\X$.

The main results of the paper are as follows. Theorem \ref{diametro} shows that $\m^\phi_\mu$ is a non-empty compact set with diameter bounded by $2\phi^{-1}(\inf F^\phi_\mu)$; the bound is tight, for instance, for invariant probability measures on compact spaces with transitive isometry group (Proposition \ref{equivariance}). Theorems \ref{Cons} and \ref{UnifCons} establish almost-sure consistency of empirical $\phi$-mean sets in one-sided Hausdorff distance for each $\phi\in\PhiEx$ for which $F^\phi_\mu$ is finite somewhere, and the corresponding uniform almost-sure consistency result for compact, uniformly growth-controlled families $\Psi\subseteq\PhiEx$ whose population means are singletons. Theorem \ref{symm} identifies $\m^\phi_\mu$ for measures on Riemannian symmetric spaces with isotropic decreasing density, and applies to every strictly increasing loss for which the objective is finite, including bounded strictly increasing losses outside $\PhiEx$.

The two consistency theorems rely on a key deterministic lemma (Lemma \ref{argmin}). It packages the standard M-estimator argument (uniform convergence on a tight compact set, together with well-separated population minima, forces one-sided Hausdorff convergence of argmin sets) in a form that applies equally to a single objective and to a family. The probabilistic inputs are a uniform LLN on compacts, and a Glivenko--Cantelli theorem for monotone function classes (Lemma \ref{glivenko}) obtained via bracketing numbers. Uniform well-separation of population minima requires a compactness argument on the index set $\Psi$, supplied by the compact-family separation lemma (Lemma \ref{epi}). Proposition \ref{chebyshev} is a deterministic application of the same argmin lemma, giving the Chebyshev limit at $p=\infty$.

The rest of the paper is organized as follows. Section \ref{SSGenFre} fixes notation and the class $\PhiEx$. Section \ref{SSCons} states the two consistency theorems. Section \ref{SSProofs} contains the abstract argmin lemma, the verification of its hypotheses, the proofs of Theorems \ref{Cons}--\ref{UnifCons}, and in Section \ref{SSChebyshev}, the Chebyshev limit is provided. Section \ref{SSSym} is the uniqueness theorem on symmetric spaces. Section \ref{SSTrees} gives a real-data illustration on projectivized ultrametric tree space. The supplementary material records the low-dimensional geometry, the Petersen-graph model of $\widetilde{\mathcal T}_4$, and the additional computational details for that application. Section \ref{SSDisc} discusses extensions and open problems. The appendix collects some auxiliary properties of $\PhiEx$, the proof of a measure-zero geometric fact used in Section \ref{SSSym}, and additional algebraic verification behind the ultrametric visualizations.

\subsection{Literature Review}\label{SSLittRev}

In linear spaces, the arithmetic mean is the standard centre, but other notions of centrality, such as the geometric median, are used to reduce sensitivity to outliers; see \citet{horror} for a discussion in linear models.  The foundational perspectives of \citet{DeFinetti} and \citet{Kolmog} remain useful in this context.  Fr\'echet means extend the arithmetic mean to Riemannian manifolds and, more generally, to metric spaces.

Large-sample theory for Fr\'echet and generalized Fr\'echet means has focused mainly on strong laws of large numbers and central limit theorems. The latter describes the rate and limiting distribution of empirical means around their population targets.  The intrinsic and extrinsic manifold theory begins with the influential work of \citet{Bhatta, Bhatta2}.  The circle has been studied in detail, including sharp conditions for uniqueness of the Fr\'echet mean \citep{circle,circle2}.  Consistency for $p$-means and related metric-space averages is treated by \citet{BhattaBook,pmeans}.  In these settings, the distance loss is fixed or belongs to a low-dimensional parametric family.

A general consistency framework is developed by \citet{Scho, Scho2}, where the loss is allowed to be an arbitrary measurable cost function and consistency of the empirical argmin set is established under a quadruple-inequality condition controlling the cost under $(x,y)\leftrightarrow(y,x)$. Our framework is complementary, as we restrict to costs of the form $\phi\circ\dd$ with convex $\phi$, but obtain explicit growth conditions ($\phi\in\PhiEx$), a sharp diameter bound for the population $\phi$-mean set, and a uniform Glivenko--Cantelli result over function classes $\Psi\subseteq\PhiEx$. To our knowledge, no analogue of the uniform statement is available for the quadruple-inequality framework. A different non-parametric extension is studied on tree spaces by \citet{tree}.

For central limit theorems, \citet{Omnibus} proved a CLT for a large class of spaces but assumed strong conditions on the support of the probability measures. Most of the hypotheses on the support of the probability measures
were removed in the recent work of \citet{CLT}, where a CLT for Fr\'echet means on Riemannian manifolds is proved. A Donsker theorem for $p$-means on manifolds also exists \citep{Donsker}. In a different direction, a CLT for Fr\'echet means was recently provided for stratified spaces \citep{Mattingly} (stratified spaces are unions and intersections of manifolds).

Tree-valued data provide one of the main non-manifold motivations for the present work.  There are several inequivalent geometries on spaces of trees; the application in this paper focuses specifically on the polyhedral space ${\mathcal T}_n$ of ultrametrics on $n$ labelled leaves, equivalently rooted equidistant phylogenetic trees, naturally represented by their pairwise leaf distances.  Its projectivized slice $\widetilde{\mathcal T}_n$ is a finite piecewise-linear stratified space described by laminar families of clusters; it is closely related to tropical and Bergman-fan models of phylogenetic tree space \citep{ArdilaKlivans2006, SpeyerSturmfels2004TropicalGrassmannian, SpeyerSturmfels2009TropicalMathematics, MonodLinYoshidaKang2018}.  This geometry is stratified, polyhedral, and finite-dimensional, but it is not treated here through nonpositive-curvature methods.  The broader asymptotic theory for singular tree and open-book spaces includes \citet{BardenLeOwen, BardenLeOwen2018, HHLMNO}, while \citet{Nye2011, NyeTangWeyenbergYoshida2017, StJohn2017} discuss visualization and dimension reduction in tree spaces.

All strong law results for generalized Fr\'echet means discussed above consider the convergence of a single loss and do not address joint consistency across a class of losses. This gap is precisely the main theoretical contribution of the present paper, providing a uniform strong-law theory for generalized Fr\'echet means under transparent and minimal structural conditions on the loss class and the sample space.

\section{Generalized Fr\'echet Means}\label{SSGenFre}
We now formalize the population and empirical minimization problems associated with the loss $\phi\circ\dd$. The point is to identify a class of losses broad enough to contain the standard centres used in statistics, while still allowing for existence, compactness, and consistency arguments on general metric sample spaces.

Throughout the paper, $(\X,\dd)$ is a metric space with the Heine--Borel property, i.e., where every closed and bounded subset is compact. Every complete, locally compact, length metric space has this property by the Hopf--Rinow theorem, stated as Theorem 2.5.28 in \citet{Gromov}; in particular, every complete connected Riemannian manifold with its geodesic metric does \citep{Lee3,doCarmo}. A Heine--Borel metric space is $\sigma$-compact and separable, so its Borel probability measures are tight. We write $\Sigma_\dd$ for the Borel $\sigma$-algebra on $\X$ and $\mathcal P(\X)$ for the set of Borel probability measures on $\X$.

We allow a class of monotone convex distance transforms that contains the Fr\'echet mean, the geometric median, and the power means. Namely, the class of admissible losses is
$$\Phi:=\{\phi:\R_+\to\R_+\text{ continuous, strictly increasing, convex, with }\phi(0)=0\},$$
equivalently the set of convex bijections of $\R_+:=[0,\infty)$.

\begin{definition}
Given $\mu\in\mathcal P(\X)$ and $\phi\in\Phi$, we define the $\phi$-mean of $\mu$ as
$$\mathfrak m^\phi_\mu:=\argmin_{x\in\X}F^\phi_\mu(x),$$
where $F^\phi_\mu:\X\to[0,\infty]$ is the $\phi$-loss function of $\mu$,
$$F^\phi_\mu(x):=\int_{\X}\phi(\dd(x,y))\mu(\ddd y).$$
\end{definition}
The choices $\phi(t)=t^2$ and $\phi(t)=t$ recover the Fr\'echet mean and the geometric median. We drop $\mu$ or $\phi$ from the notation when these are clear from context, and we identify $\m^\phi_\mu$ with its unique element when it is a singleton.
Thus $F^\phi_\mu$ is the population risk associated with the loss $\phi\circ\dd$, and $\m^\phi_\mu$ is its minimizer set. The possibility that $\m^\phi_\mu$ is set-valued is intrinsic to the theory and is crucially reflected in the consistency statements below.

Given a sample from $\mu$, the natural estimator of the population risk is its empirical plug-in version, and the resulting estimator of the centre is the empirical argmin set.
Concretely, given $\mu\in\mathcal P(\X)$ and an i.i.d.\ sample $X_1,X_2,\ldots$ of random elements from $\mu$ on $(\boldsymbol\Omega,\boldsymbol\Sigma,\boldsymbol\mu):=(\Omega^\infty,\Sigma^{\otimes\infty},\mu^{\otimes\infty})$, the empirical measure, loss and $\phi$-mean are respectively given by
$$\hat\mu_n:=\frac{1}{n}\sum_{j=1}^n\delta_{X_j},\qquad F^\phi_n(x):=\frac{1}{n}\sum_{j=1}^n\phi(\dd(X_j,x)),\qquad \m^\phi_n:=\argmin_{x\in\X}F^\phi_n(x).$$
In what follows, we write ``almost surely'' for $\boldsymbol\mu$-almost surely.

In Figure \ref{fig:conceptual-tikz-figures}, we illustrate several ways in which $\phi$-mean sets can fail to be singletons. The examples illustrate phenomena that will matter later in the paper: multiple minimizing geodesics between two observations, positive-dimensional minimizer sets, and non-uniqueness on polyhedral spaces.

\begin{figure}[!htbp]
  \centering
  \begin{minipage}{0.40\textwidth}
  \centering
    \includegraphics[width=0.7\textwidth]{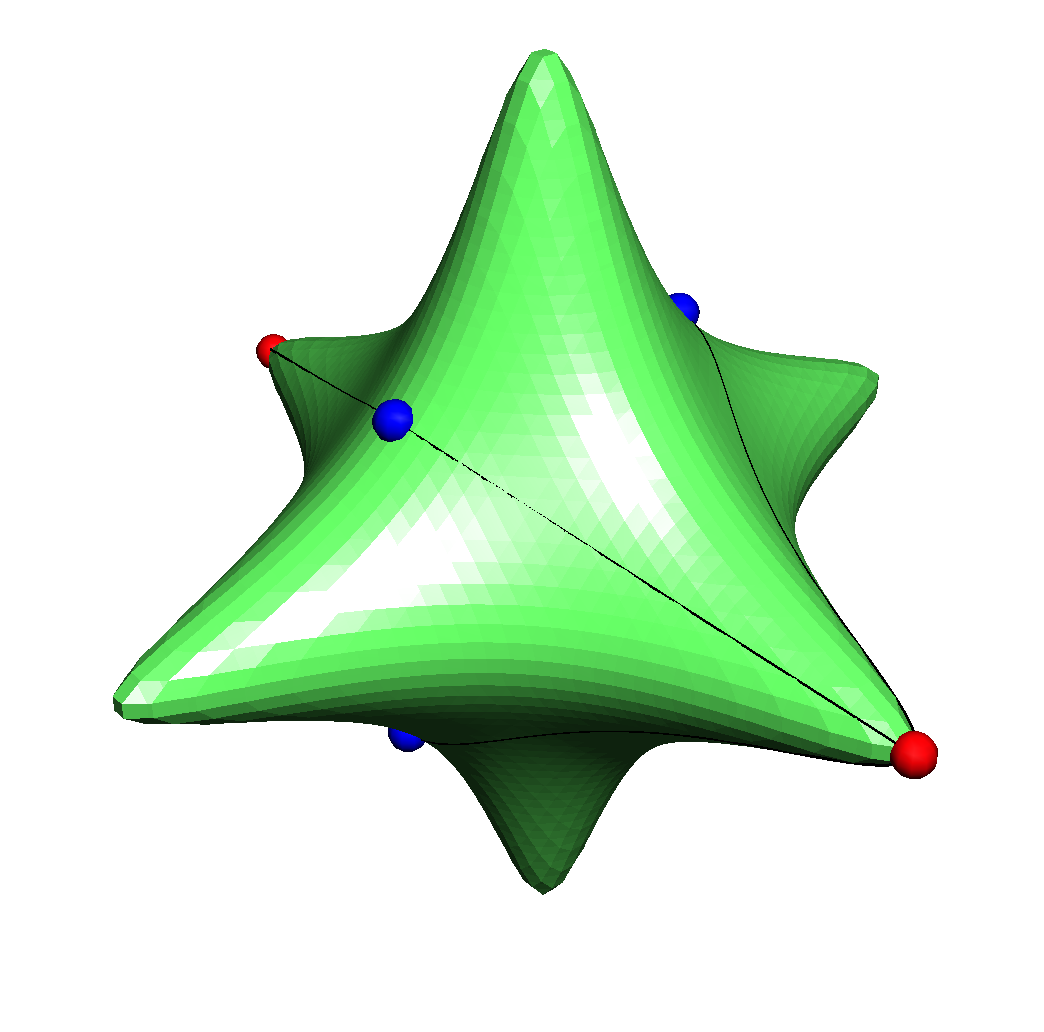}
    \\[1mm]{\small (a) The surface $\frac{x^2}{1+x^2}+\frac{y^2}{1+y^2}+\frac{z^2}{1+z^2}=0.85$ with a uniform distribution supported on the two red dots, its four minimizing geodesics and their middle points corresponding to all strictly convex $\phi$-means.}
  \end{minipage}\hfill
  \begin{minipage}{0.55\textwidth}
  \centering
    \resizebox{\textwidth}{!}{\tikzsphsetvalued}
    \\[1mm]{\small (b) Set-valued $\phi$-means on $\mathbb S^2$: equator and antipodal poles.}
  \end{minipage}\\
  \begin{minipage}{0.49\textwidth}
  \centering
    \tdplotsetmaincoords{68}{118}
    \centering
    \begin{tikzpicture}[
    tdplot_main_coords,
    scale=1.10,
    line join=round,
    line cap=round,
    ico face/.style={fill=cyan!25!white, fill opacity=0.20, draw=none},
    ico edge/.style={draw=gray!70, line width=0.55pt},
    geod/.style={draw=black, line width=0.35pt},
    antipode/.style={circle, fill=red, inner sep=1.8pt},
    midpoint/.style={circle, fill=blue, inner sep=1.35pt}
]
\def\R{1.7013016167}
\def\h{0.8506508084}
\def\H{1.9021130326}
\coordinate (A) at (0,0,\H);
\coordinate (B) at (0,0,-\H);
\foreach \k in {0,...,4} {
    \pgfmathsetmacro{\angU}{90 + 72*\k}
    \pgfmathsetmacro{\angL}{90 + 36 + 72*\k}
    \coordinate (U\k) at ({\R*cos(\angU)},{\R*sin(\angU)},\h);
    \coordinate (L\k) at ({\R*cos(\angL)},{\R*sin(\angL)},-\h);
}
\foreach \k in {0,...,4} {
    \pgfmathtruncatemacro{\kp}{mod(\k+1,5)}
    \fill[ico face] (A) -- (U\k) -- (U\kp) -- cycle;
    \fill[ico face] (U\k) -- (L\k) -- (U\kp) -- cycle;
    \fill[ico face] (L\k) -- (L\kp) -- (U\kp) -- cycle;
    \fill[ico face] (B) -- (L\kp) -- (L\k) -- cycle;
}
\foreach \k in {0,...,4} {
    \pgfmathtruncatemacro{\kp}{mod(\k+1,5)}
    \draw[ico edge] (A) -- (U\k);
    \draw[ico edge] (U\k) -- (U\kp);
    \draw[ico edge] (U\k) -- (L\k);
    \draw[ico edge] (U\kp) -- (L\k);
    \draw[ico edge] (L\k) -- (L\kp);
    \draw[ico edge] (B) -- (L\k);
}
\foreach \k in {0,...,4} {
    \pgfmathtruncatemacro{\kp}{mod(\k+1,5)}
    \pgfmathtruncatemacro{\km}{mod(\k+4,5)}
    \coordinate (Pa\k) at ($(U\k)!0.3333333333!(U\kp)$);
    \coordinate (Ma\k) at ($(U\k)!0.5!(L\k)$);
    \coordinate (Qa\k) at ($(L\k)!0.3333333333!(L\km)$);
    \draw[geod] (A) -- (Pa\k) -- (Ma\k) -- (Qa\k) -- (B);
    \coordinate (Pb\k) at ($(U\k)!0.6666666667!(U\kp)$);
    \coordinate (Mb\k) at ($(L\k)!0.5!(U\kp)$);
    \coordinate (Qb\k) at ($(L\k)!0.3333333333!(L\kp)$);
    \draw[geod] (A) -- (Pb\k) -- (Mb\k) -- (Qb\k) -- (B);
}
\foreach \k in {0,...,4} {
    \node[midpoint] at (Ma\k) {};
    \node[midpoint] at (Mb\k) {};
}
\node[antipode,label={[red!80!black]above:$A$}] at (A) {};
\node[antipode,label={[red!80!black]below:$B$}] at (B) {};
\end{tikzpicture}
    \\[1mm]{\small (c) Geodesics and $\phi$-means between two antipodal points on a regular icosahedron.}
  \end{minipage}\hfill
  \begin{minipage}{0.49\textwidth}
    \centering
\tdplotsetmaincoords{70}{120}
\begin{tikzpicture}[
    tdplot_main_coords,
    scale=1.3,
    line join=round,
    line cap=round,
    cube face/.style={fill=cyan!25!white, fill opacity=0.18, draw=none},
    cube edge/.style={draw=gray!70, line width=0.6pt},
    geod/.style={draw=black, line width=1.1pt},
    antipode/.style={circle, fill=red, inner sep=1.7pt},
    midpoint/.style={circle, fill=blue, inner sep=1.4pt}
]
\coordinate (O)   at (0,0,0);
\coordinate (X)   at (2,0,0);
\coordinate (Y)   at (0,2,0);
\coordinate (XY)  at (2,2,0);
\coordinate (Z)   at (0,0,2);
\coordinate (XZ)  at (2,0,2);
\coordinate (YZ)  at (0,2,2);
\coordinate (XYZ) at (2,2,2);
\coordinate (M1) at (1,2,0);
\coordinate (M2) at (2,1,0);
\coordinate (M3) at (2,0,1);
\coordinate (M4) at (1,0,2);
\coordinate (M5) at (0,1,2);
\coordinate (M6) at (0,2,1);
\fill[cube face] (O) -- (X) -- (XY) -- (Y) -- cycle;      
\fill[cube face] (O) -- (X) -- (XZ) -- (Z) -- cycle;      
\fill[cube face] (O) -- (Y) -- (YZ) -- (Z) -- cycle;      
\fill[cube face] (X) -- (XY) -- (XYZ) -- (XZ) -- cycle;   
\fill[cube face] (Y) -- (XY) -- (XYZ) -- (YZ) -- cycle;   
\fill[cube face] (Z) -- (XZ) -- (XYZ) -- (YZ) -- cycle;   
\draw[cube edge] (O)--(X)--(XY)--(Y)--cycle;
\draw[cube edge] (Z)--(XZ)--(XYZ)--(YZ)--cycle;
\draw[cube edge] (O)--(Z);
\draw[cube edge] (X)--(XZ);
\draw[cube edge] (Y)--(YZ);
\draw[cube edge] (XY)--(XYZ);
\draw[geod] (O) -- (M1) -- (XYZ);
\draw[geod] (O) -- (M2) -- (XYZ);
\draw[geod] (O) -- (M3) -- (XYZ);
\draw[geod] (O) -- (M4) -- (XYZ);
\draw[geod] (O) -- (M5) -- (XYZ);
\draw[geod] (O) -- (M6) -- (XYZ);
\node[antipode,label={[red!80!black]below left:$A$}] at (O)   {};
\node[antipode,label={[red!80!black]above right:$B$}] at (XYZ) {};
\node[midpoint] at (M1) {};
\node[midpoint] at (M2) {};
\node[midpoint] at (M3) {};
\node[midpoint] at (M4) {};
    \node[midpoint] at (M5) {};
    \node[midpoint] at (M6) {};
\end{tikzpicture}
    \\[1mm]{\small (d) Geodesics and $\phi$-means between two antipodal points on a regular cube.}
  \end{minipage}
\caption{Conceptual examples used throughout the discussion of generalized Fr\'echet means: nonuniqueness on a smooth surface, set-valued means on the sphere, and multiple midpoint configurations on polyhedra.}
\label{fig:conceptual-tikz-figures}
\end{figure}

\begin{remark}\label{rmkConvexity}
For every $\phi\in\Phi$ and $\mu\in\mathcal P(\X)$, $\m^\phi_\mu$ is a measure of centrality of $\mu$: we require $\phi$ increasing so that $F^\phi_\mu(x)$ grows as $x$ leaves the ``centre'' of $\mu$, and $\phi(0)=0$ is a harmless normalization. Convexity is used in every other proof, through monotonicity of secants and the midpoint inequality (both of which characterize convexity for continuous functions), and we therefore keep it as a standing assumption throughout. The exception is Theorem \ref{symm}, whose proof uses only the strict monotonicity of $\phi$ and is stated below for that strictly weaker hypothesis.
\end{remark}

The next observation shows that the population target respects symmetries of the sampling law, which is a basic mechanism used repeatedly in the sequel.
\begin{proposition}[Equivariance under isometries]\label{equivariance}
Let $(\X,\dd_\X)$ and $(\Y,\dd_\Y)$ be metric spaces and let $T:\X\to\Y$ be an isometry (i.e.\ a distance-preserving bijection). Then for every $\phi\in\Phi$ and every $\mu\in\mathcal P(\X)$,
$$F^\phi_{T_\ast\mu}\circ T=F^\phi_\mu \qquad\text{and}\qquad \m^\phi_{T_\ast\mu}=T(\m^\phi_\mu).$$
\end{proposition}
\begin{proof}
For any $x\in\X$,
\begin{align*}F^\phi_{T_\ast\mu}(Tx)&=\int_\Y\phi(\dd_\Y(Tx,y))\,(T_\ast\mu)(\ddd y)\\
&=\int_\X\phi(\dd_\Y(Tx,Tz))\,\mu(\ddd z)=\int_\X\phi(\dd_\X(x,z))\,\mu(\ddd z)=F^\phi_\mu(x),\end{align*}
where the second equality is the change of variables $y=Tz$ and the third uses that $T$ is an isometry. Taking the argmin on both sides and using that $T$ is a bijection gives $\m^\phi_{T_\ast\mu}=T(\m^\phi_\mu)$.
\end{proof}

Two corollaries follow at once. If $T:\X\to\X$ is an isometry of $\X$ and $\mu$ is $T$-invariant, then $\m^\phi_\mu$ is itself $T$-invariant. If a group $G$ acts on $\X$ by isometries and $\mu$ is $G$-invariant, then $\m^\phi_\mu$ is a $G$-invariant subset of $\X$; if in addition $G$ acts transitively, then $\m^\phi_\mu=\X$ for every $G$-invariant $\mu$ with $F^\phi_\mu$ finite. Thus, for invariant laws on spaces with a transitive isometry group, the $\phi$-mean set is forced by symmetry alone. In particular, this anticipates the tight diameter example of Theorem \ref{diametro}; indeed, equality there is not special to $S^1$, to the geometric median, or to the Haar measure, but is a consequence of equivariance alone.

The consistency theory requires a growth restriction on $\phi$ when $\X$ is unbounded. Without such a restriction, the objective $F^\phi_\mu$ may be finite only on a proper subset of $\X$, which obstructs a general statistical theory of the population and empirical minimizers.
\begin{definition}Define the two growth characteristics
$$\PhiEx:=\{\phi\in\Phi:\gamma_\phi<\infty\},\qquad \gamma_\phi:=\sup_{x\in\R_+}\frac{\phi(x+1)}{\phi(1)+\phi(x)}.$$
\end{definition}
The possibly infinite number $\gamma_\phi\in[1,\infty]$ controls the growth of $\phi$: by Lemma \ref{dominion}, $\phi\in\PhiEx$ if and only if $\phi$ is dominated by an exponential function, so we refer to $\PhiEx$ as the class of at most exponentially growing functions. Examples include $\phi(t)=t^p$ for $p\geq 1$ (with $\gamma_\phi=2^{p-1}$) and $\phi(t)=(a^t-1)/(a-1)$ for $a>1$ (with $\gamma_\phi=a$); further properties of $\PhiEx$, including alternative characterisations and stability under sums, products and positive scalar multiples, are collected in the appendix.

The next section investigates the convergence of the empirical minimizer sets towards their population counterparts almost surely, first for a fixed loss and then uniformly over a class of losses.

\section{Main Results}\label{SSCons}
For a fixed loss, the empirical $\phi$-mean can be seen as an $M$-estimator on the metric space $\X$, with population set-valued target $\m^\phi$. We first state the almost-sure consistency theorem for one loss, and then its uniform counterpart over an entire loss class.

Since the $\phi$-mean sets need not be singletons, we measure the convergence of $\m^\phi_n$ to $\m^\phi$ by the one-sided Hausdorff distance
$$\dd_\infty(A,B):=\sup_{a\in A}\inf_{b\in B}\dd(a,b)=\sup_{a\in A}\dd(a,B),$$
defined for compact subsets $A,B$ of $\X$. The functional $\dd_\infty$ is asymmetric, with $\dd_\infty(A,B)<\varepsilon$ in particular being equivalent to $A$ being contained in the open $\varepsilon$-thickening of $B$.
This asymmetric notion is the natural one here. When $\m^\phi$ is not a singleton, the statistical requirement is that empirical minimizers become asymptotically trapped near the population minimizer set; one does not require every population minimizer to be represented by the empirical argmin set at each finite sample size.

\begin{theorem}\label{Cons}
Let $(\X,\dd)$ be a metric space with the Heine--Borel property, $\phi\in\PhiEx$, and $\mu\in\mathcal P(\X)$ be such that $F^\phi_\mu$ is finite somewhere. Then, as $n\to\infty$,
$$\dd_\infty\left(\m^\phi_n,\m^\phi\right)\overset{\boldsymbol\mu-\mathrm{as}}{\to}0.$$
\end{theorem}

We note that, in particular, any measurable selection $\xi_n\in\m^\phi_n$ satisfies $\dd(\xi_n,\m^\phi)\to 0$ almost surely.
Thus, the theorem controls not only the set-valued estimator $\m_n^\phi$ itself, but also every measurable empirical representative chosen from it.

Uniformity in $\phi$ requires two additional forms of control. One must prevent the objectives from developing different tail behaviour as $\phi$ varies, and one must control the oscillation of the loss class on bounded distance ranges.
The second consistency result is uniform in $\phi$ over a class $\Psi\subseteq\PhiEx$. We call $\Psi$ \emph{dominated from above} if
\begin{align*}
\sup_{\phi\in\Psi}\gamma_\phi<\infty
\end{align*}
and there exists an envelope $\psi\in\PhiEx$ such that $\phi(t)\leq\psi(t)$ for every $\phi\in\Psi$ and $t\geq0$.  The envelope condition follows, for example, from Lemma \ref{dominion} when both $\sup_{\phi\in\Psi}\gamma_\phi$ and $\sup_{\phi\in\Psi}\phi(1)$ are finite.  We call $\Psi\subseteq\Phi_{ex}$ \emph{compact} if it is compact in $\Phi_{ex}$ equipped with the compact-open topology. A canonical example is $\{t\mapsto t^p:p\in[1,P]\}$ for any fixed $P\geq 1$, which is dominated for instance by the envelope $t+t^P$.
\begin{theorem}\label{UnifCons}
Let $(\X,\dd)$ be a metric space with the Heine--Borel property, $\mu\in\mathcal P(\X)$, and $\Psi\subseteq\PhiEx$ a non-empty compact family dominated from above by an envelope $\psi\in\PhiEx$ such that $F^\psi_\mu$ is finite at one, hence every, point of $\X$.  Assume that $F^\phi_\mu$ has a unique minimizer for every $\phi\in\Psi$. Then as $n\to\infty$,
$$\sup_{\phi\in\Psi}\dd_{\infty}(\m_n^\phi,\m^\phi)\overset{\boldsymbol\mu-\mathrm{as}}{\to}0.$$
\end{theorem}
Under the uniqueness assumption, the population set $\m^\phi$ is a singleton, so $\dd_\infty(\m_n^\phi,\m^\phi)$ represents the largest distance of an empirical minimizer from the population minimizer; if the empirical minimizer is also unique, this is the usual distance (on the appropriate space) between two points. For instance, Theorem \ref{symm} below provides a large class of measures satisfying the population uniqueness hypothesis. The result is then an empirical-process analogue, where the family of cost functions $\{\phi\circ\dd:\phi\in\Psi\}$ is Glivenko--Cantelli in a sense made precise in Lemma \ref{glivenko}, and the precise conclusion is a uniform one-sided Hausdorff convergence of the corresponding argmin sets.

The next example shows why the asymmetric Hausdorff formulation in Theorem \ref{Cons} is the correct one. Even for an elementary distribution on $\mathbb S^2$, the population target may be a positive-dimensional set, so consistency can only mean that the empirical argmin set approaches that population set.

\begin{proposition}[An exact equatorial mean set on $\mathbb S^2$]\label{prop:s2-equator-meanset}
Let $N,S\in\mathbb S^2$ be antipodal points and let
\begin{align*}
\mu_{\mathrm{eq}}=\tfrac12\delta_N+\tfrac12\delta_S.
\end{align*}
For the squared loss $\phi(t)=t^2$, the $\phi$-mean set of $\mu_{\mathrm{eq}}$ is exactly the equator
\begin{align*}
E=\{x\in\mathbb S^2:d(x,N)=d(x,S)=\pi/2\}.
\end{align*}
\end{proposition}
\begin{proof}
For $x\in\mathbb S^2$, write $\theta=d(x,N)\in[0,\pi]$. Since $d(x,S)=\pi-\theta$,
\begin{align*}
F(x)=\int_{\mathbb S^2} d(x,y)^2\,d\mu_{\mathrm{eq}}(y)
=\tfrac12\bigl(\theta^2+(\pi-\theta)^2\bigr).
\end{align*}
This depends only on $\theta$ and is minimized exactly at $\theta=\pi/2$, that is, exactly on the equator $E$.
\end{proof}
We anticipate Figure \ref{fig:s2-equator-selection}, which shows the corresponding empirical behaviour when one selects a representative point from the empirical minimizing latitude circle.

\begin{remark}[A complementary polar example]
If $\nu$ denotes the uniform probability measure on the equator, then for the same squared loss $\phi(t)=t^2$ the objective is rotationally invariant and is minimized at the two antipodal poles $N$ and $S$. Thus $\mathbb S^2$ already exhibits both kinds of non-uniqueness relevant for the present paper: a one-dimensional mean set (the equator for $\mu_{\mathrm{eq}}$) and a disconnected mean set (the pair $\{N,S\}$ for $\nu$). Both configurations are sketched in panel (b) of Figure \ref{fig:conceptual-tikz-figures}.
\end{remark}

\begin{figure}[!htbp]
  \centering
  \includegraphics[width=0.60\linewidth]{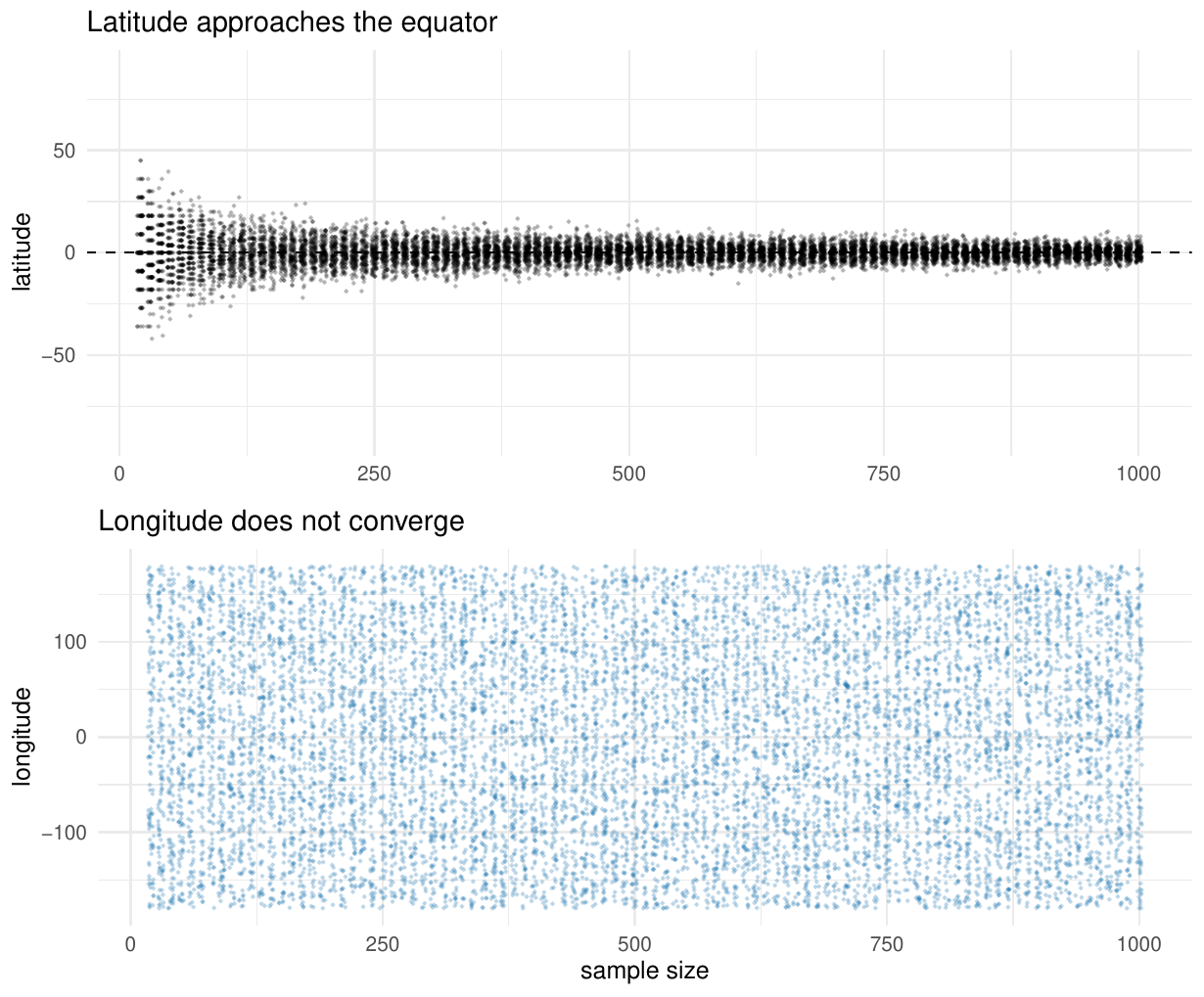}
  \caption{A set-valued consistency experiment on $\mathbb S^2$.  The underlying population distribution is $\mu_{\mathrm{eq}}=\frac12\delta_N+\frac12\delta_S$, whose squared-loss mean set is the equator by Proposition \ref{prop:s2-equator-meanset}.  For each sample size, we compute the empirical minimizing latitude circle and then select one representative point on that circle.  The latitude converges to $0$ (the equator), while the longitude does not stabilize, due to the set-valued nature of the limit.}
  \label{fig:s2-equator-selection}
\end{figure}

\section{Condition Verification and Proofs of Theorems \ref{Cons} and \ref{UnifCons}}\label{SSProofs}

The proofs are organized so that the probabilistic and deterministic parts are separated. We first show that the population criteria are finite, continuous, and sufficiently well-behaved to have compact argmin sets. We then state the key \emph{deterministic} argmin-convergence lemma (Lemma \ref{argmin}) that turns compact-local convergence of criteria into convergence of argmin sets. Finally, we verify its hypotheses, first for a fixed loss with index set $T=\{\phi\}$ for Theorem \ref{Cons}, and then uniformly over the class of losses with index set $T=\Psi$ for Theorem \ref{UnifCons}.

\subsection{Finiteness, Continuity, and Diameter}\label{SSRegularity}
We first establish that, when $\phi\in\PhiEx$, the loss $F^\phi_\mu$ is everywhere finite as soon as it is finite somewhere, is locally Lipschitz, and the $\phi$-mean set $\m^\phi_\mu$ is non-empty and compact with an explicit diameter bound. These results ensure the well-posedness required for the consistency theory. Equivalently, in the notation of Lemma \ref{argmin} below, they provide existence and compactness of the population argmin sets and the compact localization needed to keep empirical minimizers from escaping to infinity.

The arguments below use two basic inequalities for functions in $\PhiEx$, stated and proved in Lemma \ref{lemDifferences} of the appendix: for $\phi\in\PhiEx$ with $\gamma=\gamma_\phi$,
\begin{align}
\phi(x+h)-\phi(x)&\leq h\bigl(\gamma\phi(1)+(\gamma-1)\phi(x)\bigr) \quad\text{for } h\in(0,1],\label{eq:diff1}\\
2\phi(a/2)&\leq \phi(b)+\phi(c) \quad\text{whenever } a\leq b+c.\label{eq:diff2}
\end{align}
We also use repeatedly the elementary inequality: for any $\phi\in\Phi$ and $c\geq 1$, $t\geq 0$,
\begin{equation}\phi(ct)\geq c\phi(t),\label{eq:scaling}\end{equation}
which follows from convexity and $\phi(0)=0$.

\begin{lemma}\label{lemFinite}
Let $\X$ be a metric space with the Heine--Borel property.
If $\mu\in\mathcal P(\X)$ and $\phi\in\PhiEx$, then $F_\mu^\phi$ is finite somewhere iff it is finite everywhere. Conversely, if $\X$ is an unbounded, complete, locally compact, length metric space, $\phi\in\Phi$, and for every $\mu\in\mathcal P(\X)$, $F^\phi_\mu$ is finite everywhere iff it is finite somewhere, then $\phi\in\PhiEx$.
\end{lemma}
\begin{proof}
Suppose $\phi\in\PhiEx$ and $p,q\in\X$ are such that $F_\mu^\phi(p)<\infty$.  Put $\gamma=\gamma_\phi$ and $k:=\lceil\dd(p,q)\rceil$.  Iterating the inequality
\begin{align*}
\phi(a+1)\leq \gamma\{\phi(a)+\phi(1)\},
\end{align*}
with the convention used in \eqref{eq:gammak}, gives constants $C_k<\infty$ such that
\begin{align*}
\phi(a+k)\leq \gamma^k\phi(a)+C_k\qquad(a\geq0).
\end{align*}
Since $\dd(y,q)\leq\dd(y,p)+k$,
\begin{align*}
F_\mu^{\phi}(q)
\leq \int \phi(\dd(y,p)+k)\,\mu(\ddd y)
\leq \gamma^k F^\phi_\mu(p)+C_k<\infty.
\end{align*}
Thus, finiteness at one point implies finiteness everywhere.

Conversely, assume $\X$ is unbounded, complete, locally compact, and length, and let $\phi\in\Phi\setminus\PhiEx$. By the Hopf--Rinow theorem for locally compact length spaces, stated as Theorem 2.5.28 in \citet{Gromov}, $\X$ has the Heine--Borel property and contains a unit-speed geodesic ray $r:[0,\infty)\to\X$. We restrict the construction to this ray. Since $\gamma_\phi=\infty$ and
\begin{align*}
R(x):=\frac{\phi(1+x)}{\phi(1)+\phi(x)}
\end{align*}
is continuous, equals $1$ at $x=0$, and is unbounded, for every $n$ we may choose $x_n\ge0$ such that
\begin{align*}
\frac{\phi(1+x_n)}{\phi(1)+\phi(x_n)}\ge 2^{2n}.
\end{align*}
Place mass at the ray point of coordinate $1+x_n$:
\begin{align*}
\nu:=\sum_{n\ge1}2^{-n}\frac1{\phi(1)+\phi(x_n)}\delta_{r(1+x_n)},\qquad
c:=\nu(\X)\le \frac1{\phi(1)}<\infty,
\end{align*}
and set $\mu:=\nu/c$. At the point $r(1)$,
\begin{align*}
F^\phi_\nu(r(1))=\sum_{n\ge1}2^{-n}\frac{\phi(x_n)}{\phi(1)+\phi(x_n)}\le\sum_{n\ge1}2^{-n}<\infty,
\end{align*}
whereas at the point $r(0)$,
\begin{align*}
F^\phi_\nu(r(0))=\sum_{n\ge1}2^{-n}\frac{\phi(1+x_n)}{\phi(1)+\phi(x_n)}\ge\sum_{n\ge1}2^{n}=\infty.
\end{align*}
Thus, $F^\phi_\mu$ is finite somewhere but not everywhere, contradicting the assumed dichotomy. Hence $\phi\in\PhiEx$.
\end{proof}

For certain measures $\mu$, there may not exist any $\phi\in\PhiEx$ for which $F^\phi_\mu$ is finite. A typical example is $\X=\R$ and $\mu$ the standard Cauchy probability measure $\mu(\ddd x)=\ddd x/(\pi(1+x^2))$.

\begin{lemma}\label{Lipschitz}
Let $\X$ be a metric space with the Heine--Borel property and $\mu\in\mathcal P(\X)$. For all $\phi\in\PhiEx$, if $F_\mu^\phi$ is finite somewhere, then it is finite everywhere and locally Lipschitz continuous.
\end{lemma}
\begin{proof}
Global finiteness follows from Lemma \ref{lemFinite}.  Let $p,q\in\X$ with $h:=\dd(p,q)\leq1$.  If $h=0$, there is nothing to prove.  For each $x\in\X$, let $u_x\in\{p,q\}$ be a point among $p,q$ closest to $x$, and let $v_x$ be the other point.  Then $\dd(v_x,x)\leq\dd(u_x,x)+h$, and monotonicity of $\phi$ gives
$$
|\phi(\dd(p,x))-\phi(\dd(q,x))|
\leq \phi(\dd(u_x,x)+h)-\phi(\dd(u_x,x)).
$$
Applying \eqref{eq:diff1} and integrating,
\begin{align*}
|F_\mu^\phi(p)-F_\mu^\phi(q)|
&\leq h\int\bigl[\gamma_\phi\phi(1)+(\gamma_\phi-1)\phi(\dd(u_x,x))\bigr]\mu(\ddd x)\\
&\leq h\Bigl(\gamma_\phi\phi(1)+(\gamma_\phi-1)\min_{u\in\{p,q\}}F^\phi_\mu(u)\Bigr),
\end{align*}
because $\phi(\dd(u_x,x))\leq\min\{\phi(\dd(p,x)),\phi(\dd(q,x))\}$.

Fix $p_0\in\X$.  If $p,q\in B_{1/2}(p_0)$, then $\dd(p,q)\leq1$, and the one-step estimate used in Lemma \ref{lemFinite} gives
$$
F^\phi_\mu(p)\leq\gamma_\phi F^\phi_\mu(p_0)+\gamma_\phi\phi(1),
\qquad
F^\phi_\mu(q)\leq\gamma_\phi F^\phi_\mu(p_0)+\gamma_\phi\phi(1).
$$
The previous inequality therefore gives a finite Lipschitz constant on $B_{1/2}(p_0)$, depending only on $p_0$, $\phi$, and $\mu$.
\end{proof}

When $\gamma_\phi=1$ (equivalently $\phi(x)=x\phi(1)$), $F^\phi_\mu$ is globally Lipschitz continuous with constant bounded by $\phi(1)$. If $\X$ is compact, then $F^\phi_\mu$ is also globally Lipschitz. In general, it is not: for instance, if $\X=\R$ and $\mu$ is compactly supported, $F^\phi_\mu$ is globally Lipschitz iff $\phi$ is globally Lipschitz.

\begin{theorem}\label{diametro}
Let $\X$ be a metric space with the Heine--Borel property, $\mu\in\mathcal P(\X)$, and $\phi\in\PhiEx$. If $F^\phi_\mu$ is finite somewhere, then $\mathfrak m^\phi_\mu$ is a non-empty compact set with diameter bounded by
$$\mathrm{diam}(\m^\phi_\mu)\leq 2\phi^{-1}\bigl(\inf_{x\in\X}F^\phi_\mu(x)\bigr).$$
\end{theorem}
\begin{proof}
Since $\X$ has the Heine--Borel property, it suffices to show $\m^\phi_\mu$ is closed and bounded. Closedness follows from continuity of $F^\phi_\mu$ (Lemma \ref{Lipschitz}). Boundedness is clear if $\X$ is compact. If $F^\phi_\mu(p)=0$ for some $p$, then $\phi(\dd(p,x))=0$ for $\mu$-a.e. $x$, hence $\mu=\delta_p$ because $\phi$ is a bijection of $\R_+$; in this case $\m^\phi_\mu=\{p\}$ and there is nothing to prove. We may therefore assume $\X$ is unbounded and choose $p\in\X$ with $0<F^\phi_\mu(p)<\infty$. By monotone convergence, there exists $r_p>0$ with
\begin{align*}
\int_{B_{r_p}(p)}\phi(\dd(x,p))\mu(\ddd x)\geq\tfrac12 F^\phi_\mu(p).
\end{align*}
Let $q\in\X$ with $\dd(p,q)\geq 5r_p$. For every $x\in B_{r_p}(p)$,
\begin{align*}
\dd(q,x)\geq \dd(p,q)-\dd(p,x)\geq 5r_p-r_p=4r_p\geq 4\dd(p,x).
\end{align*}
By \eqref{eq:scaling} with $c=4$, $\phi(\dd(q,x))\geq 4\phi(\dd(p,x))$, and integrating,
\begin{align*}
F^\phi_\mu(q)\geq 4\int_{B_{r_p}(p)}\phi(\dd(p,x))\mu(\ddd x)\geq 2F^\phi_\mu(p)>F^\phi_\mu(p)\geq\inf F^\phi_\mu.
\end{align*}
Consequently, the infimum of $F^\phi_\mu$ over $\X$ is the same as its infimum over the compact ball $\overline{B_{5r_p}(p)}$, and is therefore attained. Moreover every minimizer lies in $\overline{B_{5r_p}(p)}$. Thus $\m^\phi_\mu$ is non-empty, closed, and bounded, hence compact.

For the diameter bound, assume $\m^\phi_\mu$ contains at least two points $y,z$. By the triangle inequality, $\dd(y,z)\leq\dd(y,x)+\dd(z,x)$ for every $x\in\X$, so by \eqref{eq:diff2} with $a=\dd(y,z)$, $b=\dd(y,x)$, $c=\dd(z,x)$,
$$2\phi(\dd(y,z)/2)\leq\phi(\dd(y,x))+\phi(\dd(z,x)).$$
Integrating against $\mu$,
$$2\phi(\dd(y,z)/2)\leq F^\phi_\mu(y)+F^\phi_\mu(z)=2\inf_{x\in\X}F^\phi_\mu(x),$$
hence $\dd(y,z)\leq 2\phi^{-1}\bigl(\inf F^\phi_\mu\bigr)$.
\end{proof}

The proof shows that on an unbounded $\X$, $F^\phi_\mu$ is unbounded: for every $p\in\X$ and $m>0$ there is a radius $s_{p,m}$ with $F^\phi_\mu(q)\geq m$ for all $q\notin B_{s_{p,m}}(p)$. Note $\inf F^\phi_\mu=0$ iff $\mu$ is a Dirac measure, in which case Theorem \ref{diametro} yields $\m^\phi_{\delta_{\hat x}}=\{\hat x\}$.

The bound is, in fact, tight. Take $\X=S^1$, $\phi=\mathrm{Id}$, and $\mu$ the normalized Haar measure. Then $F^\phi_\mu\equiv \pi/2$ is constant, so $\m^\phi_\mu=\X$ and $\diam(\m^\phi_\mu)=\diam(S^1)=\pi=2\phi^{-1}(\inf F^\phi_\mu)$.

At this point, the population optimization problem is well posed and conveniently localized on a compact region. The remaining task is to combine this localization with the convergence of the empirical criteria to obtain convergence of the empirical argmin sets.

\subsection{Abstract Argmin Convergence}\label{SSArgmin}
The following deterministic lemma underlies both consistency theorems. Its hypotheses correspond to the existence of compact population argmin sets, eventual exclusion of points outside a common compact region, uniform approximation of the criteria on that region, and strict separation of the minimizers from nearby candidates.

\begin{lemma}[Argmin convergence]\label{argmin}
Let $\X$ be a metric space with the Heine--Borel property, $T$ a non-empty index set, and $(F^t)_{t\in T}$, $(F^t_n)_{t\in T,\,n\geq 1}$ two families of continuous functions $\X\to\R$. Assume:
\begin{enumerate}
\item[(H1)] For every $t\in T$, $\m^t:=\argmin_{x\in\X}F^t(x)$ is non-empty compact, $m_t:=\inf F^t>-\infty$, and $M:=\sup_{t\in T}m_t<\infty$.
\item[(H2)] (Uniform tightness.) There is a compact $A\subseteq\X$ with $\bigcup_{t\in T}\m^t\subseteq A$ and a finite $N_0$ such that
$$\inf_{t\in T}\inf_{y\notin A}\min\bigl(F^t_n(y),F^t(y)\bigr)\geq M+1\qquad\text{for all }n\geq N_0.$$
\item[(H3)] (Uniform convergence on $A$.) $\displaystyle\sup_{t\in T}\sup_{x\in A}|F^t_n(x)-F^t(x)|\to 0$ as $n\to\infty$.
\item[(H4)] (Uniform well-separated minima.) For every $\varepsilon>0$,
$$\Delta(\varepsilon):=\inf_{t\in T}\Bigl(\inf_{\substack{x\in A\\ \dd(x,\m^t)\geq\varepsilon}}F^t(x)-m_t\Bigr)>0,$$
with the convention that $\Delta(\varepsilon)=+\infty$ if the inner set is empty.
\end{enumerate}
Then $\displaystyle\sup_{t\in T}\dd_\infty(\m^t_n,\m^t)\to 0$, where $\m^t_n:=\argmin_{x\in\X}F^t_n(x)$.
\end{lemma}
\begin{proof}
Fix $\varepsilon>0$ and set $\delta:=\min(1,\Delta(\varepsilon))/4>0$. By (H3) there is $N_1$ with $\sup_{t\in T}\sup_{x\in A}|F^t_n(x)-F^t(x)|<\delta$ for all $n\geq N_1$. For $n\geq\max(N_0,N_1)$, $t\in T$, and any $x_t^\star\in\m^t\subseteq A$ (which exists by (H1)),
$$F^t_n(x_t^\star)\leq F^t(x_t^\star)+\delta=m_t+\delta.$$
If $x\in A$ has $\dd(x,\m^t)\geq\varepsilon$ then by (H4) and (H3),
$$F^t_n(x)\geq F^t(x)-\delta\geq m_t+\Delta(\varepsilon)-\delta\geq m_t+3\delta>F^t_n(x_t^\star),$$
so $x$ is not a minimizer of $F^t_n$. If $y\notin A$ then by (H2),
$$F^t_n(y)\geq M+1\geq m_t+4\delta>F^t_n(x_t^\star),$$
so $y$ is not a minimizer either. Since $F^t_n$ is continuous and its minimum is attained on the compact set $A$, $\m^t_n$ is non-empty for all sufficiently large $n$. Moreover $\m^t_n\subseteq\{x\in A:\dd(x,\m^t)<\varepsilon\}$, i.e.\ $\dd_\infty(\m^t_n,\m^t)<\varepsilon$. The bound holds uniformly in $t\in T$, and since $\varepsilon$ was arbitrary, the result follows.
\end{proof}

\subsection{Verifying the Hypotheses}\label{SSVerify}
We now verify (H1)--(H4) of Lemma \ref{argmin} in the two cases of interest. The single-$\phi$ case ($T=\{\phi\}$) for Theorem \ref{Cons} requires only Lemmas \ref{compactGC} and \ref{compactTight} below; the uniform case ($T=\Psi$) for Theorem \ref{UnifCons} requires their uniform versions, Lemmas \ref{compactGCUnif} and \ref{compactTightUnif}, and additionally Lemmas \ref{epi}--\ref{contt} for (H4).
The order of the lemmas matches the structure of Lemma \ref{argmin}.

\begin{lemma}\label{compactGC}
For any non-empty compact set $K\subseteq\X$ and $\phi\in\PhiEx$ such that $F^\phi_\mu$ is finite at one, hence every, point of $\X$, as $n\to\infty$,
$$\sup_{x\in K}\bigl|F^\phi_n(x)-F^\phi_\mu(x)\bigr|\overset{\boldsymbol\mu-\mathrm{as}}{\to}0.$$
\end{lemma}
\begin{proof}
Set $\gamma:=\gamma_\phi$, fix $\varepsilon>0$ and a reference point $z\in K$, and let $k:=\lceil\diam(K)\rceil<\infty$. Iterating $\phi(a+1)\leq\gamma(\phi(a)+\phi(1))$ $k$ times yields the bound
\begin{equation}\phi(a+k)\leq \gamma^k\phi(a)+C_k,\qquad C_k:=\gamma\phi(1)\frac{\gamma^k-1}{\gamma-1},\label{eq:gammak}\end{equation}
valid for all $a\geq 0$ (with the convention $C_k=k\phi(1)$ when $\gamma=1$). For $x\in K$, $\dd(z,x)\leq k$, so
\begin{align*}
F^\phi_n(x)&=\frac1n\sum_{j=1}^n\phi(\dd(X_j,x))\leq \frac1n\sum_{j=1}^n\phi(\dd(X_j,z)+k)\leq \gamma^k F^\phi_n(z)+C_k.
\end{align*}
By the SLLN, $F^\phi_n(z)\to F^\phi_\mu(z)$ a.s., so there is an a.s.\ finite $N$ such that $F^\phi_n(z)\leq F^\phi_\mu(z)+1$ for $n\geq N$, and
$$\sup_{x\in K}F^\phi_n(x)\leq \alpha_K:=\gamma^k(F^\phi_\mu(z)+1)+C_k<\infty\quad\text{a.s.\ for }n\geq N.$$
For $\delta_1\in(0,1]$ and $n\geq N$, the pointwise estimate in the proof of Lemma \ref{Lipschitz} gives, for every $x,y\in K$ with $\dd(x,y)\leq\delta_1$,
\begin{align*}
|F^\phi_n(x)-F^\phi_n(y)|
&\leq \delta_1\Bigl(\gamma\phi(1)+(\gamma-1)\min\{F^\phi_n(x),F^\phi_n(y)\}\Bigr)\\
&\leq \delta_1\bigl(\gamma\phi(1)+(\gamma-1)\alpha_K\bigr).
\end{align*}
Therefore, the same bound holds after taking the supremum over such $x,y$.
Pick $\delta_1$ small enough that the right-hand side is $\leq\varepsilon/3$. By Lemma \ref{Lipschitz}, $F^\phi_\mu$ is continuous, so there is $\delta_2>0$ such that $|F^\phi_\mu(x)-F^\phi_\mu(y)|\leq\varepsilon/3$ whenever $\dd(x,y)\leq\delta_2$. Set $\delta:=\min(\delta_1,\delta_2)$ and cover $K$ by finitely many balls $B_\delta(z_1),\dots,B_\delta(z_M)$. By the SLLN, for each $i$ there is $N_i<\infty$ a.s.\ with $|F^\phi_n(z_i)-F^\phi_\mu(z_i)|\leq\varepsilon/3$ for $n\geq N_i$. Setting $N':=\max(N,N_1,\dots,N_M)$ and writing $j_x:=\min\{i:\dd(z_i,x)<\delta\}$, for $n\geq N'$ and $x\in K$,
\begin{align*}
|F^\phi_n(x)-F^\phi_\mu(x)|&\leq |F^\phi_n(x)-F^\phi_n(z_{j_x})|+|F^\phi_n(z_{j_x})-F^\phi_\mu(z_{j_x})|+|F^\phi_\mu(z_{j_x})-F^\phi_\mu(x)|\\
&\leq \varepsilon/3+\varepsilon/3+\varepsilon/3=\varepsilon.\qedhere
\end{align*}
\end{proof}

\begin{lemma}\label{compactTight}
Let $\phi\in\PhiEx$ be such that $F^\phi_\mu$ is finite at one, hence every, point of $\X$, and set $m:=\inf_{x\in\X}F^\phi_\mu(x)$. There is an a.s.\ finite random integer $N$ and a compact set $A\subseteq\X$ containing $\m^\phi$ such that for all $n\geq N$,
$$\inf_{y\notin A}\min\bigl(F^\phi_\mu(y),F^\phi_n(y)\bigr)\geq m+1.$$
\end{lemma}
\begin{proof}
If $\X$ is compact, take $A=\X$ and $N=1$; the assertion is immediate, with the usual convention that the infimum over the empty set is $+\infty$.  Otherwise, by tightness choose a compact $B\subseteq\X$ with $\mu(B)>3/4$, and let $A_k:=\{x\in\X:\dd(x,B)\leq k\}$ (compact, since closed and bounded). For $y\notin A_k$ and $x\in B$, $\dd(y,x)>k$, hence
$$F^\phi_\mu(y)\geq \int_B \phi(\dd(y,x))\mu(\ddd x)\geq \phi(k)\mu(B)\geq \tfrac{3}{4}\phi(k).$$
For the sample loss, $F^\phi_n(y)\geq \tfrac{|\{j\leq n:X_j\in B\}|}{n}\phi(k)$. By the SLLN, $|\{j\leq n:X_j\in B\}|/n\to\mu(B)>3/4$ a.s., so there is an a.s.\ finite $N$ with $|\{j\leq n:X_j\in B\}|/n>1/2$ for $n\geq N$. Choose $k^*>0$ so large that $\phi(k^*)/2>m+1$ (possible since $\phi(k)\to\infty$), and set $A:=A_{k^*}$. Then for $n\geq N$ and $y\notin A$,
$$F^\phi_\mu(y)\geq\tfrac34\phi(k^*)>m+1,\qquad F^\phi_n(y)\geq\tfrac12\phi(k^*)>m+1.$$
Finally, for $x\in\m^\phi$, $F^\phi_\mu(x)=m<m+1\leq\tfrac34\phi(k^*)$, so $\m^\phi\subseteq A$.
\end{proof}
We are now in a position to prove the first main result.
\begin{proof}[Proof of Theorem \ref{Cons}]
The pointwise theorem is obtained by checking the hypotheses of Lemma \ref{argmin} one by one.
We apply Lemma \ref{argmin} with $T=\{\phi\}$, pathwise outside a $\boldsymbol\mu$-null set. By Theorem \ref{diametro} and Lemma \ref{Lipschitz}, $\m^\phi$ is non-empty compact and $F^\phi_\mu$ is continuous; moreover $M=m=\inf F^\phi_\mu<\infty$. Thus (H1) holds. Lemma \ref{compactTight} supplies the compact set $A$ and an almost surely finite $N_0$ realising (H2). Hypothesis (H3) follows from Lemma \ref{compactGC} applied to $K=A$. For (H4) with $T=\{\phi\}$ and $\varepsilon>0$, the set $K_\varepsilon:=\{x\in A:\dd(x,\m^\phi)\geq\varepsilon\}$ is compact and disjoint from $\m^\phi$, so the continuous function $F^\phi_\mu$ attains its minimum on $K_\varepsilon$ strictly above $m=\inf F^\phi_\mu$, giving $\Delta(\varepsilon)>0$. Lemma \ref{argmin} concludes.
\end{proof}
For the uniform theorem, the compact-uniform law of large numbers and the compact localization step must both hold simultaneously over $\phi\in\Psi$. The next three lemmas provide these uniform extensions.
\begin{lemma}\label{compactGCUnif}
Let $\Psi\subseteq\PhiEx$ be non-empty and dominated from above by an envelope $\psi\in\PhiEx$ such that $F^\psi_\mu$ is finite at one point. Then for any non-empty compact $K\subseteq\X$, as $n\to\infty$,
$$\sup_{\phi\in\Psi}\sup_{x\in K}|F^\phi_n(x)-F^\phi_\mu(x)|\overset{\boldsymbol\mu-\mathrm{as}}{\to}0.$$
\end{lemma}
\begin{proof}
Set $\Gamma:=\sup_{\phi\in\Psi}\gamma_\phi<\infty$, $\eta:=\gamma_\psi$, fix $\varepsilon>0$, $z\in K$ and $k:=\lceil\diam(K)\rceil$. By the envelope condition, $\phi\leq\psi$ pointwise and hence $F^\phi_\mu\leq F^\psi_\mu$ on $\X$, and the bound on the local Lipschitz constant given by the proof of Lemma \ref{Lipschitz} yields that every $F^\phi_\mu$, $\phi\in\Psi$, is Lipschitz on $K$ with constant at most $\Gamma\psi(1)+(\Gamma-1)\sup_{x\in K}F^\psi_\mu(x)<\infty$, uniformly in $\phi$. Hence there is $\delta_1>0$ with
$$\sup_{\phi\in\Psi}\sup_{\substack{x,y\in K\\ \dd(x,y)\leq\delta_1}}|F^\phi_\mu(x)-F^\phi_\mu(y)|<\varepsilon/3.$$

For the empirical side, applying \eqref{eq:gammak} to the envelope $\psi$ gives $\phi(a+k)\leq\psi(a+k)\leq\eta^k\psi(a)+C_k^\psi$, where $
C_k^\psi=\eta\psi(1)\frac{\eta^k-1}{\eta-1}$
with the convention $C_k^\psi=k\psi(1)$ when $\eta=1$. Therefore
$$\sup_{\phi\in\Psi}\sup_{x\in K}F^\phi_n(x)\leq \eta^k F^\psi_n(z)+C_k^\psi\leq \alpha_K:=\eta^k(F^\psi_\mu(z)+1)+C_k^\psi$$
almost surely for $n$ large enough, by the SLLN applied to $F^\psi_n(z)$. The same pointwise estimate used in Lemma \ref{compactGC} then gives
$$\sup_{\phi\in\Psi}\sup_{\substack{x,y\in K\\ \dd(x,y)\leq\delta_2}}|F^\phi_n(x)-F^\phi_n(y)|\leq\delta_2\bigl(\Gamma\psi(1)+(\Gamma-1)\alpha_K\bigr)<\varepsilon/3$$
for $\delta_2$ small enough.

Set $\delta:=\min(\delta_1,\delta_2)$ and cover $K$ by finitely many balls $B_\delta(z_1),\dots,B_\delta(z_M)$. By Lemma \ref{glivenko} the class $\{\phi\circ\dd(z_i,\cdot):\phi\in\Psi\}$ is $\mu$-Glivenko--Cantelli for each $i$, so $\sup_{\phi\in\Psi}|F^\phi_n(z_i)-F^\phi_\mu(z_i)|\to 0$ a.s. Combining the three estimates by the same triangle inequality as in Lemma \ref{compactGC} gives the result.
\end{proof}

\begin{lemma}\label{glivenko}
Let $\mu\in\mathcal P(\X)$ and $\Psi\subseteq\PhiEx$ be non-empty and dominated from above by an envelope $\psi\in\PhiEx$ such that $F^\psi_\mu$ is finite at one point. Then for any $z\in\X$, the class of functions $\mathcal F:=\{f_\phi:x\mapsto\phi(\dd(z,x)):\phi\in\Psi\}$ is $\mu$-Glivenko--Cantelli.
\end{lemma}
\begin{proof}
Let $\nu$ denote the pushforward of $\mu$ under the map $y\mapsto\dd(z,y)$ from $\X$ to $\R_+$. Then
$$\|\phi\circ\dd(z,\cdot)\|_{L^1(\mu)}=\int_\X \phi(\dd(z,x))\mu(\ddd x)=\int_{\R_+}\phi(r)\nu(\ddd r),$$
so an $L^1(\nu)$-bracketing of $\Psi$ pulls back to an $L^1(\mu)$-bracketing of $\mathcal F$ with the same bracket widths. Thus it suffices, by Theorem 3.2 in \citet{Bodhisattva}, to show $N_{[]}(\varepsilon,\Psi,L^1(\nu))<\infty$ for every $\varepsilon>0$.

Since $\|\psi\|_{L^1(\nu)}=F^\psi_\mu(z)<\infty$, there is $a>0$ with $\int_{(a,\infty)}\psi(r)\nu(\ddd r)<\varepsilon/2$. Set $b:=\psi(a)$. Every $\phi\in\Psi$ is increasing from $[0,a]$ to $[0,b]$. Bracketings of the set $\mathcal L_{a,b}$ of increasing functions from $[0,a]$ to $[0,b]$ in $L^1(\nu|_{[0,a]})$ extend to $\varepsilon$-brackets of $\Psi$ in $L^1(\nu)$ by setting the lower bound to $0$ and the upper bound to $\psi$ on $(a,\infty)$. By a grid argument as in Theorem 2.7.5 in \citet{Vaart}, $N_{[]}(\varepsilon/2,\mathcal L_{a,b},L^1(\nu|_{[0,a]}))<\infty$, and the conclusion follows.
\end{proof}

\begin{lemma}\label{compactTightUnif}
Let $\Psi\subseteq\PhiEx$ be non-empty and dominated from above by an envelope $\psi\in\PhiEx$ such that $F^\psi_\mu$ is finite at one point, and non-degenerate, i.e.\ $\phi_*:=\inf_{\phi\in\Psi}\phi(1)>0$. Set $\lambda:=\inf F^\psi_\mu$. There is an a.s.\ finite random integer $N$ and a compact set $A\subseteq\X$ containing $\bigcup_{\phi\in\Psi}\m^\phi$ such that for all $n\geq N$,
$$\inf_{\phi\in\Psi}\inf_{y\notin A}\min\bigl(F^\phi_\mu(y),F^\phi_n(y)\bigr)\geq \lambda+1.$$
\end{lemma}
\begin{proof}
By Theorem \ref{diametro} applied to $\psi$, $\lambda$ is finite and attained; fix $x_\psi\in\m^\psi$.  If $\X$ is compact, take $A=\X$ and $N=1$; the assertion is immediate, again with the convention that the infimum over the empty set is $+\infty$.  Otherwise choose a compact $B\subseteq\X$ with $\mu(B)>3/4$, set $A_k:=\{x:\dd(x,B)\leq k\}$, and let $\varphi(t):=\inf_{\phi\in\Psi}\phi(t)$.

For $y\notin A_k$, $x\in B$, and $\phi\in\Psi$: $\dd(y,x)>k$, so $\phi(\dd(y,x))\geq \phi(k)\geq\varphi(k)$. Hence
$$F^\phi_\mu(y)\geq \int_B\varphi(\dd(y,x))\mu(\ddd x)\geq\varphi(k)\mu(B)\geq \tfrac34\varphi(k).$$
By convexity with $\phi(0)=0$, $\phi(t)/t$ is non-decreasing on $(0,\infty)$, so $\phi(t)\geq t\phi(1)\geq t\phi_*$ for $t\geq 1$, hence $\varphi(t)\geq t\phi_*\to\infty$.

Similarly, $F^\phi_n(y)\geq \tfrac{|\{j\leq n:X_j\in B\}|}{n}\varphi(k)$. By the SLLN, $|\{j\leq n:X_j\in B\}|/n>1/2$ a.s.\ for $n$ large.

Choose $k^*>0$ so large that $\varphi(k^*)/2>\lambda+1$, and set $A:=A_{k^*}$. Then for $n$ large,
$$F^\phi_\mu(y)\geq\tfrac34\varphi(k^*)>\lambda+1,\qquad F^\phi_n(y)\geq\tfrac12\varphi(k^*)>\lambda+1\quad\text{for all }y\notin A,\,\phi\in\Psi.$$
Finally, for each $\phi\in\Psi$, using $\phi\leq\psi$ pointwise,
$$\inf_{x\in\X} F^\phi_\mu(x)\;\leq\; F^\phi_\mu(x_\psi)\;\leq\; F^\psi_\mu(x_\psi)\;=\;\lambda<\lambda+1\leq \tfrac34\varphi(k^*),$$
so $\m^\phi$ cannot meet $A^c$, i.e.\ $\m^\phi\subseteq A$.
\end{proof}

It remains to verify (H4) uniformly in $\phi\in\Psi$. This is where compactness of $\Psi$ enters and where the uniqueness of each population minimizer is crucial. Intuitively, one needs a lower bound on the separation gap that does not deteriorate as $\phi$ varies in the set $\Psi$.

\begin{lemma}\label{contt}
Let $\Psi\subseteq\PhiEx$ be compact and dominated from above by an envelope $\psi\in\PhiEx$ such that $F^\psi_\mu$ is finite at one point. Then for any non-empty compact $A\subseteq\X$, the family $\mathcal F:=\{F^\phi_\mu|_A:\phi\in\Psi\}$ is a compact subset of $C(A)$.
\end{lemma}
\begin{proof}
We show that $\phi\mapsto F^\phi_\mu|_A$ is continuous from $\Psi$ (with the compact-open topology) to $C(A)$ (with the sup norm); the conclusion follows since the continuous image of a compact set is compact.

Fix $\varepsilon>0$ and choose $z\in A$. Put $D:=\mathrm{diam}(A\cup\{z\})$ and $k_D:=\lceil D\rceil$. Iterating the defining growth inequality for $\psi\in\PhiEx$ gives constants $C_D,c_D<\infty$ (depending only on $k_D$ and $\psi$) such that
\begin{align*}
\sup_{x\in A}\psi(\dd(x,y))\le C_D\psi(\dd(z,y))+c_D\qquad\text{for all }y\in\X.
\end{align*}
Since $F^\psi_\mu(z)<\infty$, the right-hand side is integrable. Hence, by dominated convergence applied to the decreasing sets
\begin{align*}
E_R:=\{y:\dd(z,y)>R-D\},
\end{align*}
we may choose $R>D$ such that
\begin{align*}
\sup_{x\in A}\int_{\{y:\dd(x,y)>R\}}\psi(\dd(x,y))\,\mu(\ddd y)<\varepsilon/4.
\end{align*}
For $\phi_1,\phi_2\in\Psi$ and $x\in A$,
\begin{align*}
|F^{\phi_1}_\mu(x)-F^{\phi_2}_\mu(x)|
&\leq\int_{\dd(x,y)\leq R}|\phi_1-\phi_2|(\dd(x,y))\,\mu(\ddd y)+\int_{\dd(x,y)>R}|\phi_1-\phi_2|(\dd(x,y))\,\mu(\ddd y)\\
&\leq \|\phi_1-\phi_2\|_{C([0,R])}+2\cdot\varepsilon/4,
\end{align*}
uniformly in $x\in A$, because $\phi_1,\phi_2\leq\psi$. Hence if $\|\phi_1-\phi_2\|_{C([0,R])}<\varepsilon/2$ then $\sup_{x\in A}|F^{\phi_1}_\mu(x)-F^{\phi_2}_\mu(x)|<\varepsilon$. This proves continuity of $\phi\mapsto F^\phi_\mu|_A$; since $\Psi$ is compact in the compact-open topology, its image in $C(A)$ is compact.
\end{proof}

\begin{lemma}\label{epi}
Let $A$ be compact and let $\mathcal F\subseteq C(A)$ be compact in the sup norm. Assume every $F\in\mathcal F$ has a unique minimizer $x_F\in A$. Then for every $\varepsilon>0$,
\begin{align*}
\inf_{F\in\mathcal F}\Bigl(\inf_{\substack{x\in A\\ \dd(x,x_F)\geq\varepsilon}}F(x)-F(x_F)\Bigr)>0,
\end{align*}
with the convention that the inner infimum is $+\infty$ if the displayed set is empty.
\end{lemma}
\begin{proof}
Suppose the conclusion fails for some $\varepsilon>0$. Then there are $F_j\in\mathcal F$ and $x_j\in A$ such that $\dd(x_j,x_{F_j})\ge\varepsilon$ and
\begin{align*}
F_j(x_j)-F_j(x_{F_j})\longrightarrow 0.
\end{align*}
By compactness, after passing to a subsequence, $F_j\to F$ uniformly in $C(A)$ and $x_j\to x\in A$. Uniform convergence, uniqueness of minimizers, and the elementary argmin-continuity argument imply $x_{F_j}\to x_F$: otherwise a subsequential limit $z$ of $x_{F_j}$ would satisfy $F(z)=F(x_F)$ and hence $z=x_F$. Passing to the limit gives $\dd(x,x_F)\ge\varepsilon$, while uniform convergence gives
\begin{align*}
F(x)-F(x_F)=\lim_j\{F_j(x_j)-F_j(x_{F_j})\}=0,
\end{align*}
contradicting the uniqueness of the minimizer of $F$. Hence, the infimum is strictly positive.
\end{proof}

\begin{proof}[Proof of Theorem \ref{UnifCons}]
We apply Lemma \ref{argmin} with $T=\Psi$, pathwise outside a $\boldsymbol\mu$-null set. Hypothesis (H1) holds by Theorem \ref{diametro} and the uniqueness assumption; moreover $m_\phi:=\inf F^\phi_\mu\leq\inf F^\psi_\mu$ for every $\phi\in\Psi$, so $M:=\sup_{\phi\in\Psi}m_\phi<\infty$. The evaluation $\phi\mapsto\phi(1)$ is continuous on $\Psi$ in the compact-open topology and is strictly positive on every $\phi\in\Phi$ (which is a bijection of $\R_+$); by compactness of $\Psi$, $\phi_*:=\inf_{\phi\in\Psi}\phi(1)>0$. Lemma \ref{compactTightUnif} then yields a compact $A_0$ containing $\bigcup_{\phi\in\Psi}\m^\phi$; enlarging to $A:=\{x:\dd(x,A_0)\leq 1\}$, retains compactness. Since $A^c\subseteq A_0^c$, the same tightness bound gives (H2) for this enlarged $A$. Hypothesis (H3) is Lemma \ref{compactGCUnif} applied to $K=A$. For (H4), Lemma \ref{contt} shows that $\mathcal F:=\{F^\phi_\mu|_A:\phi\in\Psi\}$ is compact in $C(A)$, and Lemma \ref{epi} supplies the uniform positive lower bound on the well-separation gap. Thus, an application of Lemma \ref{argmin} concludes.
\end{proof}

\section{Power-mean asymptotics and the Chebyshev limit}\label{SSChebyshev}
The same deterministic Lemma \ref{argmin} can also be used beyond sampling. Once the criteria converge uniformly on a compact localization set and the limit argmin set is well separated, the same mechanism also yields deterministic limits along the loss parameter.

For a measure $\mu\in\mathcal P(\X)$ with compact support $K:=\mathrm{supp}(\mu)$, define the \emph{covering radius} and the \emph{Chebyshev centre (set)} of $\mu$ as
$$R(x):=\sup_{y\in K}\dd(x,y),\qquad \m^\infty_\mu:=\argmin_{x\in\X}R(x).$$
Note that $R$ is $1$-Lipschitz and satisfies $R(x)\geq\dd(x, K)\to\infty$ as $x$ diverges to infinity, so $\m^\infty_\mu$ is both non-empty and compact, by the Heine--Borel and continuity properties.

\begin{proposition}\label{chebyshev}
Let $\X$ be a Heine--Borel metric space, $\mu\in\mathcal P(\X)$ with compact support $K$, and $\phi_p(t):=t^p$ for $p\geq 1$. Then, as $p\to\infty$,
$$\dd_\infty\bigl(\m^{\phi_p}_\mu,\,\m^\infty_\mu\bigr)\to 0.$$
\end{proposition}
\begin{proof}
Set $r_p(x):=F^{\phi_p}_\mu(x)^{1/p}=\|\dd(x,\cdot)\|_{L^p(\mu)}$. Since $t\mapsto t^p$ is strictly increasing on $\R_+$, $$\argmin r_p=\argmin F^{\phi_p}_\mu=\m^{\phi_p}_\mu,$$ so it suffices to show $\dd_\infty(\argmin r_p,\argmin R)\to 0$ as $p\to\infty$. We verify the deterministic version of Lemma \ref{argmin} with index $T=\{*\}$, $F^*=R$, $F^*_n:=r_{p_n}$ for an arbitrary sequence $p_n\to\infty$.

\emph{(H1):} $\m^\infty_\mu=\argmin R$ is non-empty compact, as noted above, and $M=m_\infty:=\inf R<\infty$.

\emph{(H2) Tightness.} For every $x\in\X$ and every $p\geq 1$, $r_p(x)\geq\dd(x,K)$ (since $\dd(x,y)\geq\dd(x,K)$ for all $y\in K$), and similarly $R(x)\geq\dd(x,K)$. Take $A:=\{y\in\X:\dd(y,K)\leq m_\infty+1\}$. Then $A$ is closed bounded, hence compact, contains $\m^\infty_\mu$ (because $R(z)=m_\infty$ implies $\dd(z,K)\leq m_\infty$), and for every $y\notin A$ and every $p$, $\min(r_p(y),R(y))\geq\dd(y,K)>m_\infty+1$, so (H2) holds with $N_0=1$.

\emph{(H3) Uniform convergence on $A$.} Fix $\varepsilon>0$. The bound $\dd(x,y)\leq R(x)$ for $y\in K$ gives $r_p(x)\leq R(x)$. For the converse direction, set
$$K_\varepsilon(x):=\{y\in K:\dd(x,y)>R(x)-\varepsilon\}.$$
This is the intersection of $K=\supp\mu$ with an open set of $\X$ and contains a maximizer of $\dd(x,\cdot)|_K$, so $\mu(K_\varepsilon(x))>0$. The map $x\mapsto \mu(K_\varepsilon(x))$ is lower semicontinuous: if $x_n\to x$ and $y\in K_\varepsilon(x)$, the continuity of $R$ and $\dd$ gives $\dd(x_n,y)>R(x_n)-\varepsilon$ for all large $n$, so $K_\varepsilon(x)\subseteq\liminf_n K_\varepsilon(x_n)$, and Fatou's lemma for sets yields $\mu(K_\varepsilon(x))\leq\liminf_n\mu(K_\varepsilon(x_n))$. Lower semicontinuity on the compact set $A$ implies that $\delta_\varepsilon:=\inf_{x\in A}\mu(K_\varepsilon(x))$ is attained and strictly positive. Hence
$$r_p(x)^p\geq\int_{K_\varepsilon(x)}\dd(x,y)^p\,\mu(\ddd y)\geq (R(x)-\varepsilon)_+^p\mu(K_\varepsilon(x))\geq (R(x)-\varepsilon)_+^p\delta_\varepsilon,$$
where $u_+:=\max(u,0)$. Thus $r_p(x)\geq (R(x)-\varepsilon)_+\delta_\varepsilon^{1/p}$ uniformly in $x\in A$. Since $R$ is bounded on $A$ and $\delta_\varepsilon^{1/p}\to1$, it follows that $\limsup_{p\to\infty}\sup_{x\in A}|R(x)-r_p(x)|\leq\varepsilon$. As $\varepsilon>0$ is arbitrary, $r_p\to R$ uniformly on $A$.

\emph{(H4) Separation.} Continuity of $R$ and compactness of $\m^\infty_\mu$ give, for every $\varepsilon>0$, $\Delta(\varepsilon):=\inf\{R(x)-m_\infty:x\in A,\,\dd(x,\m^\infty_\mu)\geq\varepsilon\}>0$.

Lemma \ref{argmin} then yields $\dd_\infty(\argmin r_{p_n},\m^\infty_\mu)\to 0$ along any sequence $p_n\to\infty$, hence as $p\to\infty$.
\end{proof}

The preceding statement is qualitative and applies in arbitrary Heine--Borel spaces. In one elementary model, however, the whole path can be computed in closed form. The exact formula below isolates the deterministic interpolation in $p$, and we subsequently provide two numerical illustrations, one in the same model and one on $\mathbb S^2$.

\begin{proposition}[Explicit two-point power mean]\label{prop:two-point-phi-p}
Let $L>0$, let $a\in(0,1)$, and consider the probability measure
\begin{align*}
\mu=a\,\delta_0+(1-a)\,\delta_L
\end{align*}
on the interval $[0,L]$ with the Euclidean distance. For $p\ge1$, set
\begin{align*}
F_p(t):=a|t|^p+(1-a)|L-t|^p,\qquad t\in[0,L].
\end{align*}
If $p>1$, then $F_p$ has the unique minimizer
\begin{align*}
 t_p=
 \frac{(1-a)^{1/(p-1)}}{a^{1/(p-1)}+(1-a)^{1/(p-1)}}\,L, \quad\mbox{equivalently},\quad
\frac{t_p}{L-t_p}=\left(\frac{1-a}{a}\right)^{1/(p-1)}.
\end{align*}
For $p=1$, the minimizer is $0$ if $a>1/2$, is $L$ if $a<1/2$, and is the whole interval $[0,L]$ if $a=1/2$. Moreover, if $a\ne1/2$, then $t_p\to0$ as $p\downarrow1$ when $a>1/2$, while $t_p\to L$ as $p\downarrow1$ when $a<1/2$, and in all cases $t_p\to L/2$ as $p\to\infty$.
\end{proposition}

\begin{proof}
For $p>1$, strict convexity of $t\mapsto t^p$ on $[0,\infty)$ makes $F_p$ strictly convex on $[0,L]$. The unique minimizer is characterized by
\begin{align*}
F'_p(t)=ap t^{p-1}-(1-a)p(L-t)^{p-1}=0,
\end{align*}
which is equivalent to
$
a t^{p-1}=(1-a)(L-t)^{p-1}$.
Solving this equation gives the displayed formula. For $p=1$,
$F_1(t)=at+(1-a)(L-t)=(2a-1)t+(1-a)L,$
so the minimizer set is immediate. The limits follow directly from a careful analysis of the explicit formula.
\end{proof}

As $p$ grows, $\m^{\phi_p}_\mu$ moves from the geometric median ($p=1$) toward the Chebyshev centre of $\supp(\mu)$. The deterministic map $(a,p)\mapsto\m^{\phi_p}_{a\delta_0+(1-a)\delta_1}$ in the above model is shown in Supplementary Material, Section \ref{AASpherePath}, Figure \ref{fig:placeholder}. Theorem \ref{UnifCons} gives empirical consistency uniformly on compact exponent ranges on which the population p-mean is unique, while Proposition \ref{chebyshev} gives the deterministic endpoint as $p\to\infty$. Also note that no stochastic uniformity over the non-compact index set $[1,\infty]$ is required here.

Further numerical illustrations are deferred to Supplementary Material, Section \ref{AASpherePath}. Figure \ref{fig:s2-consistency} illustrates Theorems \ref{Cons} and \ref{UnifCons} on a compact range of exponents, whereas Figure \ref{fig:s2-trajectory} shows the corresponding empirical path on $\mathbb {S} ^2$, where the target remains nonlinear, but the progression from median-like to minimax-like centres is still visible.

\section{Isotropic Densities on Symmetric Spaces and Uniqueness}\label{SSSym}
The uniform strong law in Theorem \ref{UnifCons} requires the population $\phi$-mean to be a singleton for every $\phi$ in the loss class. The purpose of the present section is to verify that hypothesis on a broad model class and, more strongly, to show that for isotropic decreasing laws on symmetric spaces, the population target is actually independent of the strictly increasing loss.

Concretely, we aim to provide an explicit uniqueness result for isotropic measures whose $\phi$-mean is a singleton for every monotone $\phi$ in the stated class. The uniqueness question has been studied in special cases. For instance, \citet{circle, circle2} give sharp criteria, and asymptotics for the Fr\'echet mean on the circle, \citet{Andy} treats Fr\'echet means more generally, and generalized Fr\'echet-type means also appear in shape-statistical work such as \citet{Huckemann2011Procrustes}.

Throughout the section $\X$ is a connected Riemannian manifold equipped with its geodesic distance $\dd$, and we further assume that $\X$ is a \emph{symmetric space}: a complete Riemannian manifold such that for any two distinct points $x,y\in\X$ there is an isometric involution $R_{x,y}$ of $\X$ exchanging $x$ and $y$. Symmetric spaces were classified in the classical work of \citet{Cartan, Cartan2, Helga}; they include $\R^n$, the spheres, the real, complex, and quaternionic projective spaces, Grassmannians, all compact semi-simple Lie groups, and any Cartesian product of these, so for instance the torus $T^n$.

The measure $\mu$ is taken absolutely continuous with respect to the volume measure on $\X$ \citep{Lee3}. We consider isotropically decreasing probability densities, that is,
$$\mu_{f,\hat x}(\ddd x)=f(\dd(x,\hat x))\,\vol(\ddd x),$$
where $\hat x\in\X$, the normalizing constant is included in $f$, and $f:[0,\infty)\to[0,\infty)$ is non-negative and strictly decreasing on $[0,\diam(\X)]$, with the convention $\diam(\X)=\infty$ when $\X$ is unbounded.
The proof of the uniqueness result at the end of this section uses a reflection argument. For a competitor $y\neq\hat x$, one compares $F^\phi_\mu(y)$ and $F^\phi_\mu(\hat x)$ by pairing points through the isometric involution exchanging $y$ and $\hat x$. The only additional geometric input is that the bisector between these two points has zero volume.

Thus, we use the following geometric fact, whose proof is given in the appendix.
\begin{lemma}\label{bisector}
Let $\X$ be a complete connected Riemannian manifold and $p,q\in\X$ two distinct points. Then the bisector
$$\X^{p,q}:=\{z\in\X:\dd(z,p)=\dd(z,q)\}$$
has volume measure zero.
\end{lemma}

\begin{theorem}\label{symm}
Let $\X$ be a symmetric space and let $\mu=\mu_{f,\hat x}$. Let $\phi:\R_+\to\R_+$ be any strictly increasing function (\emph{not} necessarily convex or in $\PhiEx$) for which $F^\phi_\mu(x)<\infty$ for every $x\in\X$. Then the unique $\phi$-mean of $\mu$ is $\m^\phi_\mu=\{\hat x\}$.
\end{theorem}
The hypothesis on $\phi$ is therefore strictly weaker than elsewhere in the paper: continuity, convexity, and the growth condition $\gamma_\phi<\infty$ are not required. Theorem \ref{symm} also covers bounded strictly increasing losses, for example $\phi(t)=1-e^{-t}$, whenever the corresponding objective is finite everywhere. Redescending losses with flat tails require an additional argument unless all relevant distances lie in a region in which the loss is strictly increasing.
\begin{proof}
Let $y\in \X\setminus\{\hat x\}$ and let $R:=R_{\hat x,y}$ be the involution exchanging $\hat x$ and $y$. Split $\X=\X^+\sqcup \X^0\sqcup \X^-$, where
$$\X^+:=\{x\in \X:\dd(x,\hat x)>\dd(x,y)\},\quad \X^0:=\{x:\dd(x,\hat x)=\dd(x,y)\},\quad \X^-:=\{x:\dd(x,\hat x)<\dd(x,y)\}.$$
Since $R$ is an isometry and exchanges $\hat x$ and $y$, $\dd(Rx,\hat x)=\dd(x,y)$ and $\dd(Rx,y)=\dd(x,\hat x)$. Therefore $R$ maps $\X^+$ bijectively to $\X^-$ and preserves $\X^0$. Moreover, $R$ preserves the volume measure (as an isometry).

By Lemma \ref{bisector}, $\vol(\X^0)=0$, so this piece contributes nothing to $F^\phi_\mu(\hat x)-F^\phi_\mu(y)$. Performing the change of variables $x'=Rx$ on the $\X^-$ piece,
\begin{align*}
F^\phi_\mu(\hat x)-F^\phi_\mu(y)&=\int_{\X^+}\bigl(\phi(\dd(x,\hat x))-\phi(\dd(x,y))\bigr)f(\dd(x,\hat x))\vol(\ddd x)\\
&\quad+\int_{\X^+}\bigl(\phi(\dd(Rx,\hat x))-\phi(\dd(Rx,y))\bigr)f(\dd(Rx,\hat x))\vol(\ddd x)\\
&=\int_{\X^+}\bigl[\phi(\dd(x,\hat x))-\phi(\dd(x,y))\bigr]\bigl[f(\dd(x,\hat x))-f(\dd(Rx,\hat x))\bigr]\vol(\ddd x),
\end{align*}
where in the second integrand we used $\phi(\dd(Rx,\hat x))=\phi(\dd(x,y))$ and $\phi(\dd(Rx,y))=\phi(\dd(x,\hat x))$, so the bracket becomes the negative of the first one.

For $x\in\X^+$ the first bracket is strictly positive (since $\dd(x,\hat x)>\dd(x,y)$ and $\phi$ is strictly increasing) and the second is non-positive (since $\dd(Rx,\hat x)=\dd(x,y)<\dd(x,\hat x)$ and $f$ is non-increasing), so the integrand is pointwise $\leq 0$ on $\X^+$.

Since $y\in\X^+$ and $\X^+$ is open, $\X^+$ contains a non-empty geodesic ball and hence has positive volume. Under the strict radial monotonicity assumption on $f$, the second bracket is strictly negative throughout $\X^+$, while the first bracket is strictly positive. Hence, the paired integrand is strictly negative on a set of positive volume, so $F^\phi_\mu(\hat x)<F^\phi_\mu(y)$. Since $y\in\X\setminus\{\hat x\}$ was arbitrary, $\m^\phi_\mu=\{\hat x\}$.
\end{proof}
In particular, the above result provides us with a broad class of measures for which the uniqueness hypothesis of Theorem \ref{UnifCons} is automatic. Within this class, changing the loss affects the empirical criterion but not the population centre.

Without the symmetry hypothesis, the conclusion actually fails: an isotropic measure $\mu_{f,\hat x}$ on a general manifold need not have a unique $\phi$-mean, and the $\phi$-mean set need not contain $\hat x$.

The next section gives a real-data illustration on the tree space (intended as the space of all ultrametrics on a finite set). In contrast with the isotropic decreasing regime treated here, that singular sample space exhibits loss-dependent behaviour and changes of stratum along the intrinsic $\phi_p$-mean path.

\section{Application on the space of rooted trees}\label{SSTrees}

The preceding results apply to any metric space satisfying the Heine--Borel property.  We now give a real-data illustration on the singular polyhedral space of ultrametrics, equivalently rooted equidistant phylogenetic trees, on a fixed set of labelled leaves.  The statistical point of this example is that the sample space is stratified and the intrinsic $\phi_p$-mean need not remain in a fixed stratum as $p$ varies.  Detailed geometry of the projectivized ultrametric complexes, the Petersen-graph model of $\widetilde{\mathcal T}_4$, and the computational construction of the intrinsic approximation are deferred to the supplementary material, especially Sections \ref{SSTrees-ultrametric}--\ref{SSTrees-ultra-computation}.

Let ${\mathcal T}_n$ denote the projectivized space of ultrametrics on $n$ labelled leaves, equipped with the intrinsic path metric $\rho$ induced by the Euclidean metric on each polyhedral cell.  For data $x_1,\ldots,x_k\in \widetilde{\mathcal T}_n$ and $p\ge1$, the intrinsic empirical objective is
\begin{equation}\label{eq:ultra-objective}
        F_{k,p}^{{\mathcal T}_n}(q)
        :=\frac1k\sum_{j=1}^k \rho(q,x_j)^p,
        \qquad q\in \widetilde{\mathcal T}_n .
\end{equation}
This is the statistic to which the consistency theorems apply.  It differs from ambient Euclidean cophenetic averaging followed by UPGMA or another projection step.

\subsection{The Apicomplexa gene-tree data}\label{SSTrees-data}

The data used in this illustration are the Apicomplexa gene-tree sample of \citet{Kuo08}.  In phylogenetics, \emph{taxa} are the labelled organisms, species, or sequences placed at the leaves of a tree; here, the original data consist of $268$ rooted gene trees inferred from single-copy genes for seven apicomplexan species and one ciliate outgroup.  Apicomplexans include medically and agriculturally important parasitic protists, such as \emph{Plasmodium}, \emph{Cryptosporidium}, \emph{Toxoplasma}, \emph{Babesia}, and \emph{Eimeria}. Different genes can support different evolutionary histories.  Summarizing a collection of gene trees is therefore not only a numerical averaging problem: it is also a way to describe the dominant pattern of gene-tree incongruence.

The full eight-taxon tree space is already too large for the exact low-dimensional visualization used in this paper.  We therefore use the four taxa $\mathrm{Bb},\ \mathrm{Cp},\ \mathrm{Et},\ \mathrm{Pf}$, namely \emph{Babesia bovis}, \emph{Cryptosporidium parvum}, \emph{Eimeria tenella}, and \emph{Plasmodium falciparum}, and regard the resulting computation as an auditable four-taxon illustration.  For each restricted gene tree, we form the six cophenetic distances $(d_{12},d_{13},d_{14},d_{23},d_{24},d_{34})$.  Since the numerical branch lengths supplied in the data need not give an exactly ultrametric set of distances (which would correspond naturally to a rooted equidistant tree), the implementation projects these vectors to the four-taxon ultrametric cone ${\mathcal T}_4$ before computing the centres. 
The computed conical centres are then projectivized only for display in Figure \ref{fig:ultra-pu4-embedding}. The projection and mesh construction are described in Section~\ref{SSTrees-ultra-computation}.

The relevant sample space ${\mathcal T}_4$ is a \emph{stratified space}: informally, it is obtained by gluing ordinary Euclidean pieces, called strata, along lower-dimensional faces. In this tree-space example, a top-dimensional stratum corresponds to a resolved rooted ultrametric topology, while a lower-dimensional stratum corresponds to an unresolved tree in which one or more internal lengths have collapsed. Thus, several different resolved topologies meet along the same lower-dimensional face.  This is exactly the kind of singular geometry in which an intrinsic centre can change not only its numerical coordinates but also its combinatorial tree type as the loss function is varied; see also the general tree-space background in \citet{GAVUSHKIN, StJohn2017, MillerOwenProvan2015}.

Figure \ref{fig:ultra-pu4-embedding} shows the restricted sample and selected nested-mesh approximations to the intrinsic $\phi_p$-means.  The grey points are the projectivized restricted gene trees after projection to ${\mathcal T}_4$, and the coloured points are the corresponding projectivized $\phi_p$-means for several values of $p$.  The path should be interpreted as a geometric summary of the four-taxon gene-tree distribution.  Smaller values of $p$ emphasize the central bulk of the sample, whereas larger values give more weight to distant trees and therefore move the centre toward the more extreme parts of the observed distribution.
In the present numerical approximation, the displayed centres lie close to a one-dimensional stratum in the projectivized space.  This suggests that, after restriction to these four taxa, much of the visible variation is organized along a dominant direction of topological and branch-length disagreement.

The main statistical feature is that the loss parameter can affect the active stratum of the representative tree. In biological terms, this means that different choices of $p$ can lead to different compromises among the observed gene histories. The example illustrates how intrinsic loss-based centres can be used as diagnostic summaries for tree-valued data: they reveal which parts of the gene-tree distribution are stable under changes of loss and which parts are driven by more distant or conflicting gene trees. Selected dendrogram renderings of the computed centres are reported in Figure \ref{fig:ultra-means} of the supplementary material.

\begin{figure}[!htbp]
  \centering
  \includegraphics[width=0.88\linewidth]{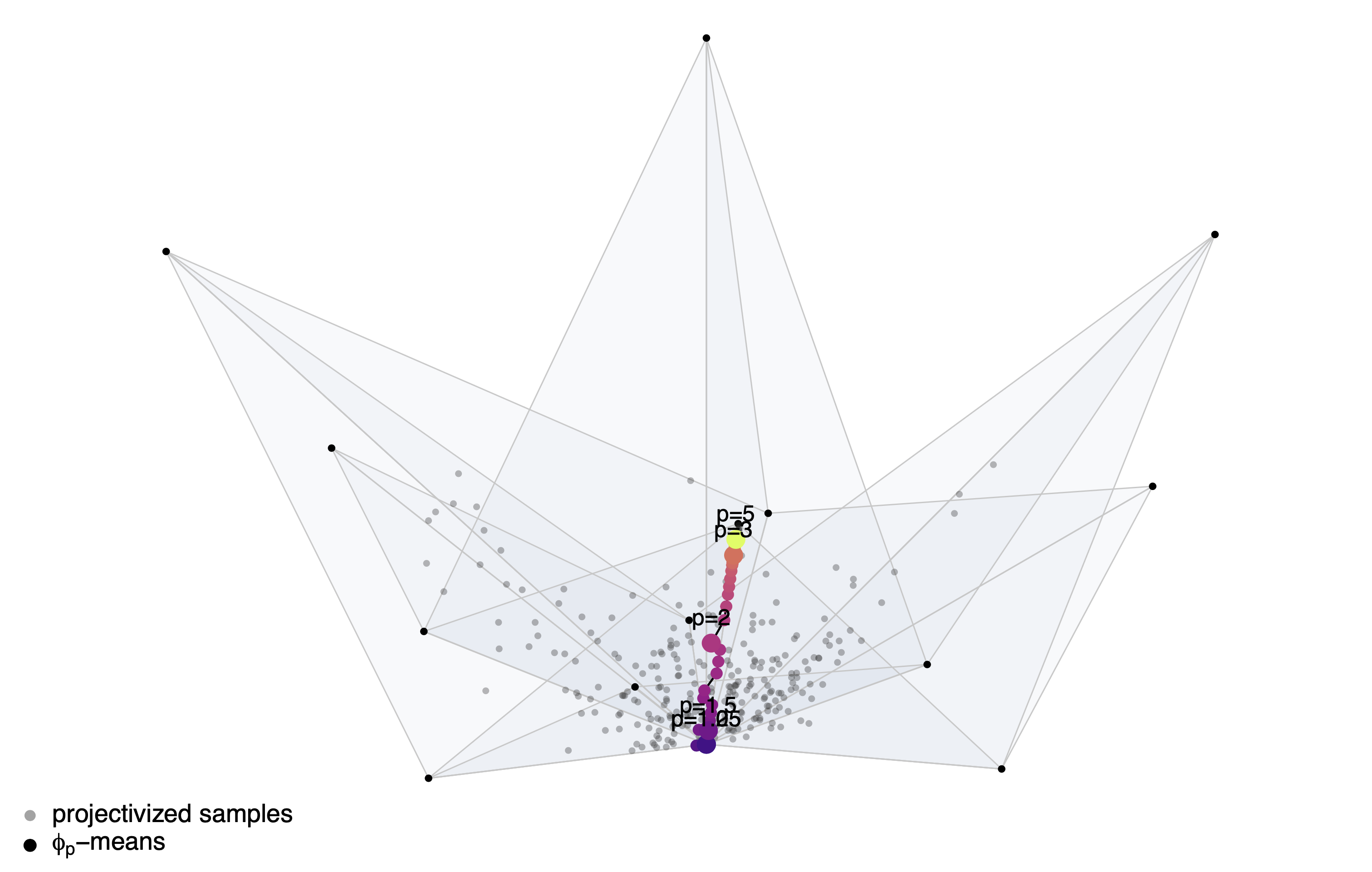}
  \caption{Restricted Apicomplexa sample and selected nested-mesh $\phi_p$-means in a cellwise isometric three-dimensional visualization of $\widetilde{\mathcal T}_4$.  Grey points are the projectivized restricted gene trees after projection to the ultrametric cone ${\mathcal T}_4$; coloured points are the projectivized conical $\phi_p$-means.  The figure is a four-taxon geometric summary of gene-tree incongruence, not a replacement for a full eight-taxon species-tree analysis.}
  \label{fig:ultra-pu4-embedding}
\end{figure}

\section{Discussion and Open Problems}\label{SSDisc}

Theorems \ref{Cons} and \ref{UnifCons} are convergence results under rather lax conditions on very general spaces. Extending them to asymptotic normality is not a formal corollary, since it requires second-order regularity at the population minimizer together with local control of the argmin geometry. Under such hypotheses, for instance, a positive-definite Hessian of $F^\phi_\mu$ on a Riemannian manifold or an analogous quadratic lower bound on a general metric space \citep{Scho2}, standard $M$-estimator theory, such as Theorem 5.23 of \citet{Vaart}, one expects asymptotic normality for each fixed $\phi$. On bounded spaces and for Lipschitz losses, one may likewise expect exponential concentration for the empirical argmin sets, but deriving such bounds would require quantitative refinements of the separation and compact-local comparison arguments used in the proofs. A further natural question is to obtain quantitative rates in the Chebyshev limit of Proposition \ref{chebyshev}, and more generally for other one-parameter loss families. Some of these questions are currently being investigated by the authors.


\begin{acks}[Acknowledgments]\label{acks}
The authors are thankful to Giulio Principi for useful suggestions and to Shreya Arya, Andrew McCormack, Katerina Papagiannouli, Rohit Kumar, Ludovico Crippa, Benedetta Bruni, Andrea Agazzi, and Peter Hoff for discussions.  The authors are especially grateful to Xiongzhi Chen and Stephan F. Huckemann for helpful correspondence concerning the reflection argument in Theorem \ref{symm} and the measure-zero bisector lemma; in particular, Xiongzhi Chen pointed out a sign error in an earlier draft.
\end{acks}

\begin{funding}
AA and MB would like to acknowledge support from the Carlsberg Foundation, grant CF23-1096.
\end{funding}

\bibliographystyle{imsart-nameyear} 
\bibliography{PHI1}       

\newpage
\section*{Supplementary Material}
\appendix

\section{Spaces of Convex Functions}\label{AAAppendix}
Recall that $\Phi$ denotes the collection of convex bijections of $\R_+:=[0,\infty)$ (equivalently, continuous, strictly increasing, convex $\phi:\R_+\to\R_+$ with $\phi(0)=0$), and
$$\PhiEx:=\Bigl\{\phi\in\Phi:\gamma_\phi:=\sup_{x\in\R_+}\frac{\phi(x+1)}{\phi(x)+\phi(1)}<\infty\Bigr\}.$$
The function $x\mapsto\phi(1+x)/(\phi(1)+\phi(x))$ is continuous on $\R_+$, so equivalently
$$\PhiEx=\Bigl\{\phi\in\Phi:\limsup_{x\to\infty}\frac{\phi(1+x)}{\phi(1)+\phi(x)}<\infty\Bigr\}=\Bigl\{\phi\in\Phi:\limsup_{x\to\infty}\frac{\phi(1+x)}{\phi(x)}<\infty\Bigr\}.$$
Since the ratio equals $1$ at $x=0$, $\gamma_\phi\geq 1$, with equality iff $\phi$ is linear, i.e.\ $\phi(x)=x\phi(1)$. Examples: $\phi:x\mapsto x^p$ for $p\geq 1$ ($\gamma_\phi=2^{p-1}$); and $\phi:x\mapsto (p^x-1)/(p-1)$ for $p>1$ ($\gamma_\phi=p$).

Both $\Phi$ and $\PhiEx$ are closed under positive scaling, sum, and pointwise multiplication.  For products, convexity follows from midpoint convexity: if $0\leq x\leq y$, then monotonicity gives $(\phi(y)-\phi(x))(\psi(y)-\psi(x))\geq0$, which is exactly the inequality needed to show $(\phi\psi)((x+y)/2)\leq(\phi(x)\psi(x)+\phi(y)\psi(y))/2$.  Specifically, for $c>0$ and $\phi,\psi\in\PhiEx$,
$$\gamma_{c\phi}=\gamma_\phi,\qquad\gamma_{\phi\psi}\leq 2\gamma_\phi\gamma_\psi,\qquad\gamma_{\phi+\psi}\leq\max(\gamma_\phi,\gamma_\psi).$$
The first equality is clear. For the second,
\begin{align*}
\gamma_{\phi\psi}&=\sup_{x}\frac{\phi(x+1)\psi(x+1)}{\phi(x)\psi(x)+\phi(1)\psi(1)}\\
&=\sup_x \frac{\phi(x+1)\psi(x+1)}{(\phi(x)+\phi(1))(\psi(x)+\psi(1))}\cdot\frac{(\phi(x)+\phi(1))(\psi(x)+\psi(1))}{\phi(x)\psi(x)+\phi(1)\psi(1)}\\
&\leq \gamma_\phi\gamma_\psi\cdot\sup_x\Bigl(1+\frac{\phi(1)\psi(x)+\psi(1)\phi(x)}{\phi(x)\psi(x)+\phi(1)\psi(1)}\Bigr)\leq 2\gamma_\phi\gamma_\psi,
\end{align*}
where the last step uses the fact that the cross-term ratio is bounded above by $1$; indeed
\begin{align*}
\phi(x)\psi(x)+\phi(1)\psi(1)-\phi(1)\psi(x)-\psi(1)\phi(x)
=(\phi(x)-\phi(1))(\psi(x)-\psi(1))\geq0
\end{align*}
by monotonicity. For the third,
\begin{align*}
\gamma_{\phi+\psi}&=\sup_x \frac{\phi(x+1)+\psi(x+1)}{\phi(x)+\psi(x)+\phi(1)+\psi(1)}\\
&=\sup_x\Bigl(\frac{\phi(x+1)}{\phi(x)+\phi(1)}\alpha(x)+\frac{\psi(x+1)}{\psi(x)+\psi(1)}(1-\alpha(x))\Bigr)\leq\max(\gamma_\phi,\gamma_\psi),
\end{align*}
where $\alpha(x):=(\phi(x)+\phi(1))/(\phi(x)+\phi(1)+\psi(x)+\psi(1))\in[0,1]$.

\begin{lemma}\label{powergamma} For any $t>0$,
$$\PhiEx=\Bigl\{\phi\in\Phi:\gamma_{\phi}^{(t)}:=\sup_{x\in\mathbb R_+}\frac{\phi(x+t)}{\phi(x)+\phi(t)}<\infty\Bigr\}=\Bigl\{\phi\in\Phi:\limsup_{x\to\infty}\frac{\phi(x+t)}{\phi(x)}<\infty\Bigr\}.$$
\end{lemma}
\begin{proof}
We show that for $0<s<t$, finiteness of $\gamma_\phi^{(s)}$ and $\gamma_\phi^{(t)}$ are equivalent. If $\gamma_\phi^{(t)}<\infty$, then monotonicity gives
$$
\gamma_\phi^{(s)}=\sup_x\frac{\phi(x+s)}{\phi(x)+\phi(s)}
\leq \sup_x \frac{\phi(x+t)}{\phi(x)+\phi(t)}\cdot\frac{\phi(t)+\phi(x)}{\phi(s)+\phi(x)}
\leq \gamma_\phi^{(t)}\frac{\phi(t)}{\phi(s)}<\infty.
$$
Conversely, suppose $\gamma_\phi^{(s)}=\gamma<\infty$ and put $k:=\lceil t/s\rceil$. Iterating
$\phi(u+s)\le \gamma(\phi(u)+\phi(s))$ gives constants $A_k,B_k<\infty$ such that
$$
\phi(x+ks)\le A_k\phi(x)+B_k\qquad\text{for all }x\ge0.
$$
Since $t\le ks$ and $\phi$ is increasing,
$$
\frac{\phi(x+t)}{\phi(x)+\phi(t)}\le
\frac{A_k\phi(x)+B_k}{\phi(x)+\phi(t)}\le A_k+\frac{B_k}{\phi(t)}<\infty.
$$
Thus $\gamma_\phi^{(t)}<\infty$. The equivalence with the displayed limsup follows because $\phi(t)>0$ is fixed and $\phi(x)\to\infty$ as $x\to\infty$.
\end{proof}

\begin{lemma}\label{dominion}
Functions in $\PhiEx$ are dominated by exponentials. If $\phi\in\PhiEx$ with $\gamma=\gamma_\phi>1$, then
$$\phi(x)\leq \gamma\frac{\gamma^x-1}{\gamma-1}\phi(1)\qquad\forall x\geq 0.$$
For $\gamma=1$, $\phi(x)=x\phi(1)$.
\end{lemma}
\begin{proof}
The $\gamma=1$ case follows because then $\phi(x+1)\leq\phi(x)+\phi(1)$ for all $x$. Convexity gives the reverse inequality for the unit increment, $\phi(x+1)-\phi(x)\geq\phi(1)-\phi(0)=\phi(1)$, so equality holds for every unit interval. Since the right derivative of a convex function is non-decreasing, equality of all unit increments forces that derivative to be constant; hence $\phi(x)=x\phi(1)$.

Assume $\gamma>1$. We prove by induction on $\lfloor x\rfloor$ that
$$\phi(x)\leq\phi(1)\sum_{j=1}^{\lfloor x\rfloor}\gamma^j+\gamma^{\lfloor x\rfloor}\phi(x-\lfloor x\rfloor).$$
For $\lfloor x\rfloor=0$, the inequality reduces to $\phi(x)\leq\phi(x)$. Suppose the bound holds for $\lfloor x\rfloor=n$. For $y$ with $\lfloor y\rfloor=n+1$ and $x:=y-1$,
\begin{align*}
\phi(y)=\phi(x+1)&\leq\gamma(\phi(x)+\phi(1))\\
&\leq\gamma\Bigl(\phi(1)+\phi(1)\sum_{j=1}^{\lfloor x\rfloor}\gamma^j+\gamma^{\lfloor x\rfloor}\phi(x-\lfloor x\rfloor)\Bigr)\\
&=\phi(1)\sum_{j=1}^{\lfloor y\rfloor}\gamma^j+\gamma^{\lfloor y\rfloor}\phi(y-\lfloor y\rfloor).
\end{align*}
By convexity and $\phi(0)=0$, $\phi(x-\lfloor x\rfloor)\leq(x-\lfloor x\rfloor)\phi(1)$. Together with $\phi(1)\sum_{j=1}^{\lfloor x\rfloor}\gamma^j=\gamma\phi(1)(\gamma^{\lfloor x\rfloor}-1)/(\gamma-1)$,
$$\phi(x)\leq\frac{\gamma\phi(1)}{\gamma-1}(\gamma^{\lfloor x\rfloor}-1)+\gamma^{\lfloor x\rfloor}(x-\lfloor x\rfloor)\phi(1).$$
Using the elementary inequality $y\leq g(g^y-1)/(g-1)$ for $g>1$, $y\in[0,1]$ (equivalently, $g^y-1\geq y(g-1)/g$, which follows from $\log g\geq (g-1)/g$ and convexity of $u\mapsto g^u$), we conclude
$$\phi(x)\leq \gamma\frac{\gamma^x-1}{\gamma-1}\phi(1).$$
\end{proof}

\begin{lemma}\label{lemDifferences}
Let $\phi\in\PhiEx$ with $\gamma=\gamma_\phi$. Then:
\begin{enumerate}
\item For any $h\in(0,1]$ and $x\geq0$,
$$\phi(x+h)-\phi(x)\leq h\bigl(\gamma\phi(1)+(\gamma-1)\phi(x)\bigr).$$
\item If $a,b,c\geq 0$ satisfy $a\leq b+c$, then $2\phi(a/2)\leq \phi(b)+\phi(c)$.
\end{enumerate}
\end{lemma}
\begin{proof}
(i) By definition of $\gamma_\phi$, $\phi(x+1)\leq \gamma(\phi(x)+\phi(1))$, so $\phi(x+1)-\phi(x)\leq \gamma\phi(1)+(\gamma-1)\phi(x)$. By convexity, for $h\in(0,1]$,
$$\phi(x+h)-\phi(x)\leq h\bigl(\phi(x+1)-\phi(x)\bigr)\leq h\bigl(\gamma\phi(1)+(\gamma-1)\phi(x)\bigr).$$

(ii) If $b\ge a$, then $\phi(b)\ge\phi(a)\ge2\phi(a/2)$ by monotonicity and convexity with $\phi(0)=0$, and the claim follows; the same argument applies if $c\ge a$. We may therefore assume $b<a$ and $c<a$. From $a\le b+c$ we get $a-b\le c$. By convexity,
\begin{align*}
2\phi(a/2)\le \phi(b)+\phi(a-b),
\end{align*}
and by monotonicity $\phi(a-b)\le\phi(c)$. Hence $2\phi(a/2)\le\phi(b)+\phi(c)$.
\end{proof}

As a consequence,
$$\limsup_{x\to\infty}\frac{\phi'_+(x)}{\phi(x)}\leq \limsup_{x\to\infty}\frac{(\gamma_\phi-1)\phi(x)+\gamma_\phi\phi(1)}{\phi(x)}=\gamma_\phi-1.$$

\section{Bisectors on Riemannian Manifolds}\label{AABisector}
We give a self-contained proof of Lemma \ref{bisector}. The result is folklore and can also be deduced from more general statements about subanalytic sets or from the theory of sets with positive reach developed by \citet{Bangert}. A concise proof for complete Riemannian manifolds was also given in the MathOverflow discussion \citep{MathOverflowBisector}. The completeness hypothesis is essential for this formulation: the same discussion gives a non-complete Riemannian manifold for which a bisector has positive volume.

\begin{proof}[Proof of Lemma \ref{bisector}]
Let $n:=\dim\X$. If $n=1$ then a complete connected one-dimensional Riemannian manifold is isometric, up to scale, to $\R$ or to $S^1$, and the bisector of two distinct points consists of at most two points, hence has volume zero.

Assume $n\geq 2$. For $a\in\X$ let $C_a$ denote the cut locus of $a$, which is a closed set with $\vol(C_a)=0$; see \citet[Proposition III.4.1]{Sakai} and \citet[Chapter 11]{BishopCrittenden}. On the open set $\X\setminus(\{a\}\cup C_a)$, the function $\dd(\cdot,a)$ is smooth and its gradient $\nabla\dd(\cdot,a)(z)$ is the unit tangent vector at $z$ of the unique minimizing geodesic from $a$ to $z$, pointing away from $a$; in particular $\|\nabla\dd(\cdot,a)\|\equiv 1$ on this set.

Set
$$U:=\X\setminus\bigl(\{p,q\}\cup C_p\cup C_q\bigr),$$
which is open with $\vol(\X\setminus U)=0$, and define $h:U\to\R$ by $h(z):=\dd(z,p)-\dd(z,q)$. Then $h$ is smooth on $U$ and
$$\|\nabla h(z)\|^2=2-2\langle\nabla\dd(z,p),\nabla\dd(z,q)\rangle.$$
We claim that $0$ is a regular value of $h$. Indeed, if $h(z)=0$ and $\nabla h(z)=0$, then the two unit vectors $\nabla\dd(z,p)$ and $\nabla\dd(z,q)$ coincide. Reversing the two unique minimizing geodesics from $z$ to $p$ and from $z$ to $q$, this says that $p$ and $q$ lie on the same minimizing geodesic ray starting at $z$. Since $p\ne q$, one is strictly farther from $z$ than the other, contradicting $h(z)=0$. Thus $\nabla h\ne0$ on $h^{-1}(0)\cap U$.

By the regular value theorem, $h^{-1}(0)\cap U$ is locally a smooth hypersurface, hence has volume zero. Combining with $\vol(\X\setminus U)=0$ gives
$$\vol\bigl(\X^{p,q}\bigr)=\vol\bigl(h^{-1}(0)\cap U\bigr)+\vol\bigl(\X^{p,q}\cap(\X\setminus U)\bigr)=0.$$
\end{proof}

\section{Supplementary figures}\label{AAOutlier}\label{AASpherePath}
Figure \ref{fig:s2-two-point-outlier} shows the dependence of the empirical displacement on the loss exponent in a one-outlier contamination experiment on $\mathbb S^2$. It complements the discussion of loss-indexed robustness in Section \ref{SSChebyshev}.

\begin{figure}[!htbp]
  \centering
  \includegraphics[width=0.65\linewidth]{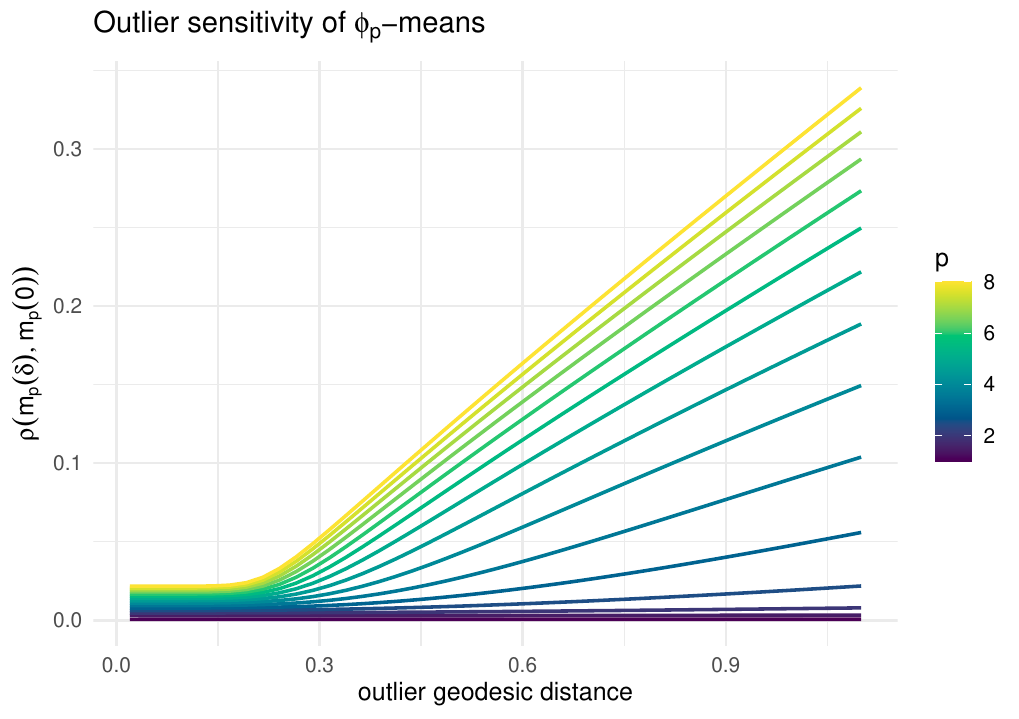}
  \caption{The displacement caused by one outlying observation as a function of its geodesic distance from a concentrated spherical cluster, for several exponents $p$.}
  \label{fig:s2-two-point-outlier}
\end{figure}

Figures \ref{fig:placeholder} complement the discussion in Section \ref{SSChebyshev}. They show the deterministic interpolation in the explicit two-point model, the finite-range uniform consistency experiment, and the corresponding empirical $\phi_p$-mean path on $\mathbb S^2$.

\begin{figure}[!htbp]
    \centering
    \includegraphics[width=0.65\textwidth]{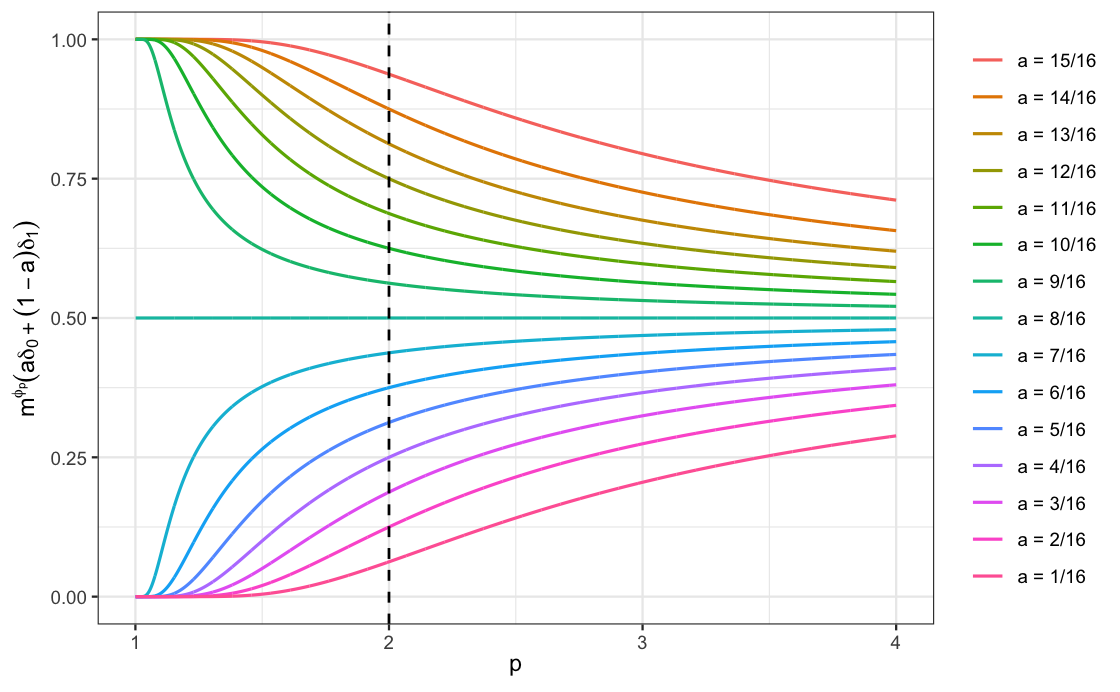}
    \caption{The $\phi_p$-mean of $a\delta_0+(1-a)\delta_1$. For $a=1/16,\ldots,15/16$ and $p\in(1,4]$. Notice that $\mathfrak m^{\phi_2}(a\delta_0+(1-a)\delta_1)=1-a$.}
    \label{fig:placeholder}
\end{figure}

\begin{figure}[!htbp]
  \centering
  \includegraphics[width=0.65\linewidth]{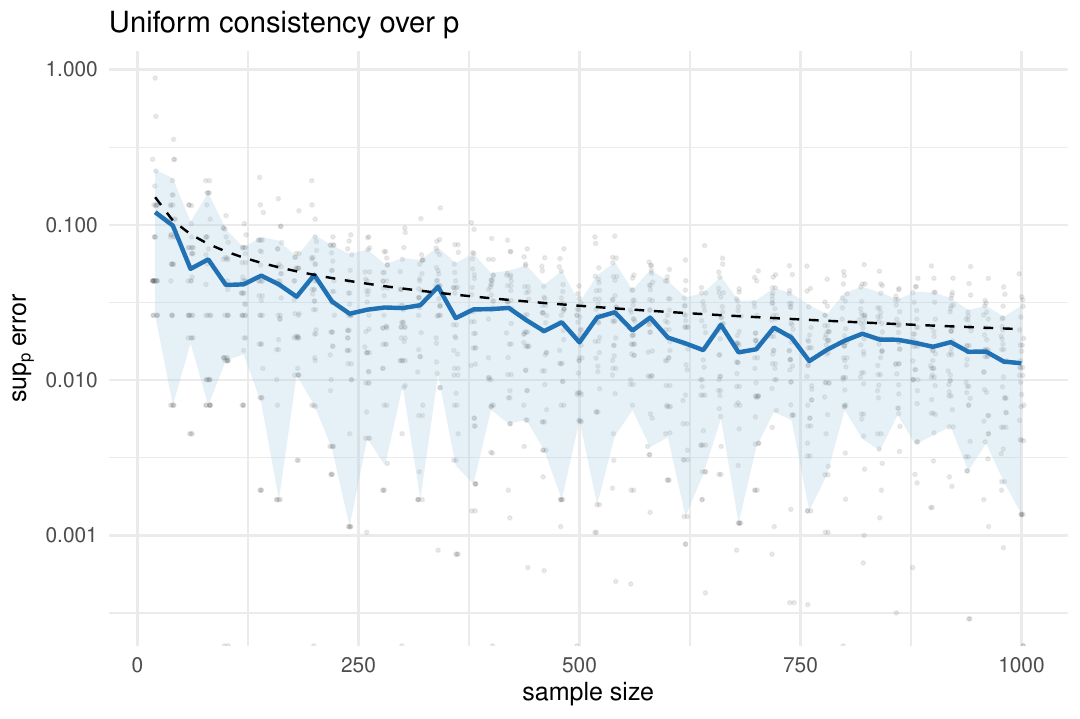}
  \caption{Consistency in the explicit two-point model of Proposition \ref{prop:two-point-phi-p}.  The left panel displays pointwise error over a fine grid $p=1.1,1.2,\ldots,4$ and sample sizes increasing by increments of ten.  The right panel plots the supremum over the same $p$-grid, with an $n^{-1/2}$ reference line.  This figure illustrates Theorems \ref{Cons} and \ref{UnifCons} in a setting where the population $\phi_p$-mean is known exactly.}
  \label{fig:s2-consistency}
\end{figure}

\begin{figure}[!htbp]
  \centering
  \includegraphics[width=0.65\linewidth]{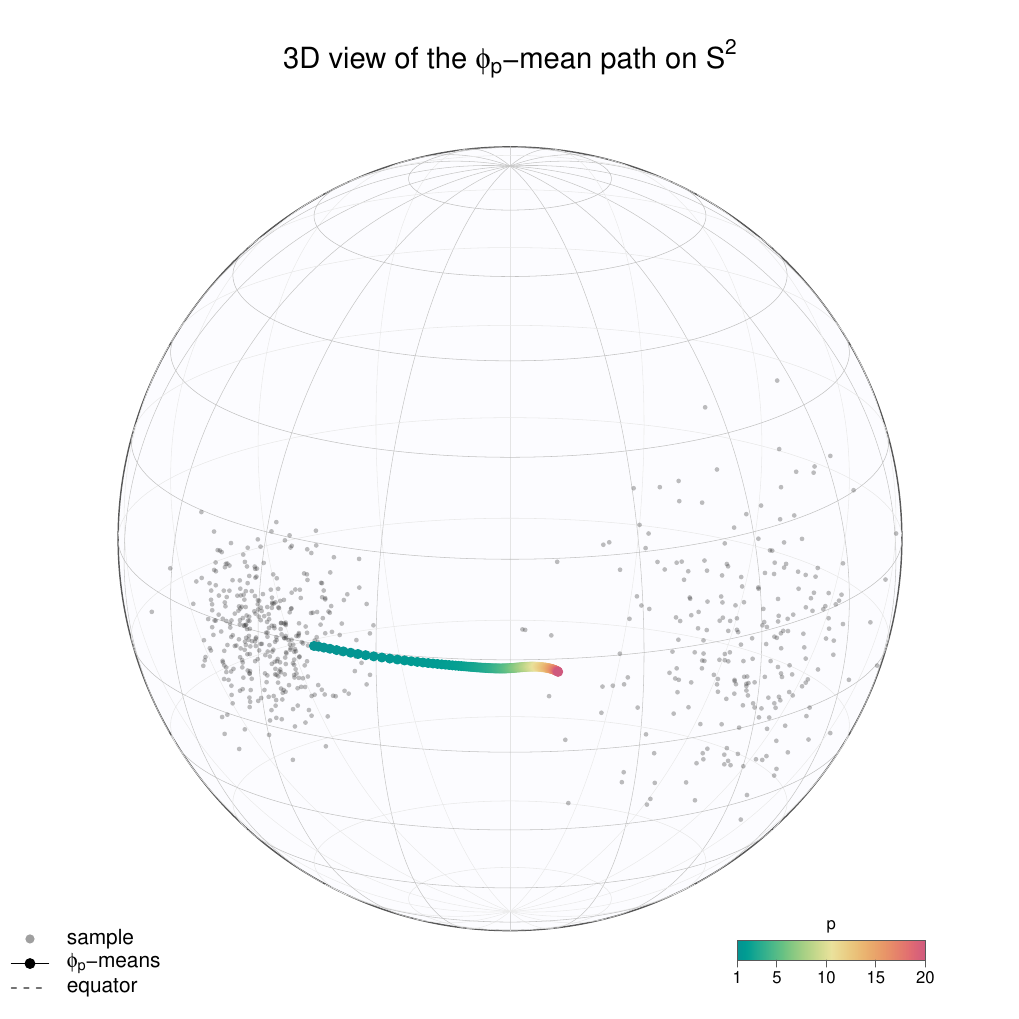}
  \caption{Equatorial two-blob simulation on $S^2$.  The two cluster centres lie on the equator and are symmetrically placed about longitude zero.  Small grey points are observations, and the coloured path is the empirical $\phi_p$-mean for a fine grid of exponents $p\in[1,20]$.  The path starts near the geometric median, passes through the Fr\'echet mean, and moves toward a minimax/Chebyshev-type centre.}
  \label{fig:s2-trajectory}
\end{figure}

\section{Geometry and computation on tree spaces}\label{AAUltrametricApplication}
This section records the low-dimensional geometry and the computational details underlying Section \ref{SSTrees}.  The main text keeps only the real-data illustration; here we collect the structure of the projectivized ultrametric complexes, the Petersen-graph model of $\widetilde{\mathcal T}_4$, the construction of the intrinsic mesh approximation, and the additional figures from the four-taxon Apicomplexa analysis.

\subsection{The ultrametric cone and its projectivization}\label{SSTrees-ultrametric}

Let $[n]=\{1,\ldots,n\}$ and put $N=\binom n2$.  A metric on $[n]$ can be identified with a vector
\begin{align*}
        d=(d_{ij})_{1\le i<j\le n}\in\R^N .
\end{align*}
The coordinate $d_{ij}=d_{ji}$ represents the distance between elements $i$ and $j$.
The collection of possible metrics on $[n]$ is the cone
$$\mathcal M_n:=\{d\in \R^N_{\geq0}: d_{ij}\leq d_{ik}+d_{kj} \text{ for all distinct }i,j,k\}.$$
The sub-cone $\mathcal T_n\subseteq\mathcal M_n$ of \emph{ultrametrics} on $[n]$ is obtained by replacing the triangular inequality with the strong triangular inequality, that is
\begin{align*}
    {\mathcal T}_n:=\left\{d\in\R_{\ge0}^N: d_{ij}\le \max\{d_{ik},d_{jk}\}\text{ for all distinct }i,j,k\right\}.
\end{align*}
Both $\mathcal M_n$ and $\mathcal T_n$ are cones whose  projectivizations (affine slices) we denote by
\begin{align*}
        \widetilde{\mathcal T}_n
        :=\left\{d\in{\mathcal T}_n:\sum_{i<j} d_{ij}=1\right\}\subseteq\left\{d\in{\mathcal M}_n:\sum_{i<j} d_{ij}=1\right\}=:\widetilde{\mathcal M}_n.
\end{align*}
The space $\mathcal M_n$ is a closed convex full-dimensional polyhedral cone in $\R^N$. Its projectivation $\widetilde{\mathcal M}_n$ is a non-uniform convex polytope of dimension $N-1$.
The spaces $\mathcal T_n$ and $\widetilde{\mathcal T}_n$ are piecewise linear (PL) sets of dimensions $n-1$ and $n-2$, respectively.  They are highly non-convex and are not topological manifolds; rather, they are PL-stratified spaces in the standard sense of piecewise Euclidean geometry.  For background on ultrametric tree spaces and CAT(0)/piecewise Euclidean geometry, see \citet{GAVUSHKIN} and \citet{BRIDSON}.
The symmetric group $S_n$ acts naturally on all these spaces by relabeling the elements of $[n]$.  For the full metric cone, \citet{Deza} prove that, for $n\ge5$, every Euclidean linear isometry preserving $\mathcal M_n$ is induced by such a relabeling; the exceptional case $n=4$ has extra ambient symmetries for the full metric cone.
We equip all these spaces with the path metric induced by the Euclidean metric on each flat cell. 

The cell structure of $\widetilde{\mathcal T}_n$ is described by laminar families\footnote{A laminar family is a collection $\mathcal F$ of sets such that if $A, B\in\mathcal F$ then only three options are possible: either $A\cap B=\emptyset$ or $A\subseteq B$ or $B\subseteq A$.}.  For $A\subseteq[n]$ with $|A|\ge2$, define a normalized ray vector $v_A\in \widetilde{\mathcal T}_n$ by
$$(v_A)_{ij}=\left(\binom n2-\binom{|A|}{2}\right)^{-1}\mathbf1_{\{i,j\}\not\subseteq A}$$
for $A\neq[n]$, and set $v_{[n]}=(1/N,\ldots,1/N)$.  If $\mathcal F$ is a laminar family of non-singleton subsets containing $[n]$, then
\begin{align*}
        \sigma_{\mathcal F}:=\operatorname{conv}\{v_A:A\in\mathcal F\}
\end{align*}
is a cell of $\widetilde{\mathcal T}_n$.  The vertex $v_{[n]}$ belongs to every maximal cell, so this description makes clear that $\widetilde{\mathcal T}_n$ is star-shaped with respect to $v_{[n]}$ (but not convex, for $n\ge3$).  Maximal cells correspond to rooted binary trees and have dimension $n-2$ in $\widetilde{\mathcal T}_n$.

The first nontrivial projectivized example is $\widetilde{\mathcal T}_3$, which is already stratified despite being one-dimensional. Figure \ref{fig:T3} identifies it with a tripod inside the projectivized three-point metric simplex and provides the local model for the higher-dimensional ultrametric spaces considered below.

\begin{figure}
    \centering
    \resizebox{0.6\textwidth}{!}{\tikzputhree}
  \caption{The standard simplex
$\Delta_3=\{(d_{1,2},d_{1,3},d_{2,3})\in\mathbb R_{\ge 0}^3:d_{1,2}+d_{1,3}+d_{2,3}=1\}$ is an equilateral triangle in $\R^3$. The set of all metric spaces on three points where the sum of the three distances is one, $\widetilde{\mathcal M}_3$, is a subset of the simplex. In particular, it is the medial equilateral triangle cut out by the three triangular inequalities $d_{1,2}\le d_{1,3}+d_{2,3}$, $d_{1,3}\le d_{1,2}+d_{2,3}$, $d_{2,3}\le d_{1,2}+d_{1,3}$.
The set of all ultrametrics on three points where the sum of the three distances is one, $\widetilde{\mathcal T}_3\subseteq \widetilde{\mathcal M}_3$ is the tripod $\{d_{1,2}=d_{1,3}\ge d_{2,3}\}\cup\{d_{1,2}=d_{2,3}\ge d_{1,3}\}\cup\{d_{1,3}=d_{2,3}\ge d_{1,2}\}$,
whose three outer vertices correspond to the degenerate rooted trees where leaves are at zero distance, and whose central vertex is the tree whose root has degree three and whose leaves are all at distance $1/6$ from the root. The maximal strata correspond to maximal laminar families and to the rooted binary trees. For instance $\mathrm{conv}\{(0,1,1)/2,(1,1,1)/3\}=\mathrm{conv}\{v_{\{1,2\}},v_{\{1,2,3\}}\}$ corresponds to the maximal laminar family $\{\{1,2\},\{1,2,3\}\}$ and the binary tree in the picture. Accidentally, the projectivized space of trees on three nodes is a tree with three nodes.}
  \label{fig:T3}
\end{figure}

\subsection{The projectivized four-leaf space}\label{SSTrees-PU4}

The first two-dimensional case is $\widetilde{\mathcal T}_4$.  It is a two-dimensional polyhedral complex with one distinguished vertex $v_{1234}$, ten one-dimensional strata, and fifteen two-dimensional cells.  The ten one-dimensional strata are indexed by the nontrivial proper clusters
\begin{align*}
  12,13,14,23,24,34,123,124,134,234.
\end{align*}
Two such clusters are compatible if they are nested or disjoint.  The compatibility graph on these ten clusters is the Petersen graph, and the fifteen edges of that graph correspond exactly to the fifteen two-dimensional cells of $\widetilde{\mathcal T}_4$.  Hence $\widetilde{\mathcal T}_4$ is the cone over the Petersen graph, with apex $v_{1234}$.

\begin{figure}[!htbp]
  \centering
  \resizebox{0.82\textwidth}{!}{\tikzpufourpetersen}
  \caption{The projectivized ultrametric space $\widetilde{\mathcal T}_4$, as the cone over the Petersen graph. The centre is the fully unresolved equidistant tree $v_{1234}$. The ten boxed vertex glyphs represent the one-dimensional strata, indexed by the nontrivial proper clusters of $\{1,2,3,4\}$; the fifteen boxed edge glyphs represent the two-dimensional cells, indexed by compatible cluster pairs.  This picture is a model of the ultrametric/equidistant rooted-tree space.}
  \label{fig:pu4-petersen}
\end{figure}

Figure \ref{fig:pu4-petersen} makes explicit the stratified geometry relevant for the $\phi$-means computations. In $\widetilde{\mathcal T}_4$, a geodesic may pass through edges, vertices, or two-dimensional cells, and the same phenomenon becomes combinatorially richer for $\widetilde{\mathcal T}_5$ and beyond.

\subsection{Computational aspects and additional figures for the Apicomplexa example}\label{SSTrees-ultra-computation}

In general, to compute $\phi$-means on $\mathcal T_n$, we first need to compute distances between points, that is, the lengths of the minimizing geodesics connecting the two points. This can be done efficiently since, on each stratum, geodesics are just straight lines. Then, to compute the minimizing geodesic between points on different strata, we determine the optimal cell path and then solve a standard convex optimization problem to determine the optimal crossing points between adjacent cells. Concretely, we consider many possible cell paths and then minimize over these possible paths.

Once distances between points are available, in order to find the $\phi$-means of a sample $(x_j)_{j=1}^S$, we proceed by looking for minimizers of $x\mapsto \sum_{j=1}^S \phi(\dd(x,x_j))$ on a fine grid of candidate points covering $\mathcal T_n$. Once a minimizer has been found among the points of this grid, a further, finer grid localized around this first candidate can be constructed, and the operation can be repeated.

The code we used for Section \ref{SSTrees-data} provides a numerical approximation to the continuous optimization problem of finding $\phi_p$-means on $\mathcal T_4$. It is carried out only for the four-taxon three-dimensional ultrametric cone $\mathcal T_4$, whose projectivization is two-dimensional and can be isometrically immersed in $\R^3$ for an effective visual inspection.

The input trees are loaded from the \texttt{apicomplexa} object in the R package \texttt{Rtropical} \citep{RtropicalPackage}. The code then restricts each tree to the chosen four taxa and forms its six-coordinate cophenetic vector $(d_{12},d_{13},d_{14},d_{23},d_{24},d_{34})\in\R^6.$
After restriction, these raw vectors need not be exactly ultrametric. The default preprocessing therefore projects each raw vector to the nearest point of $\mathcal T_4$, using active-set least squares over the eighteen three-ray cone cells.

The expensive step is the evaluation of intrinsic distances between sample points and candidate centres. The code exploits the fact that $\mathcal T_4$ is a Euclidean metric cone.  If $u$ and $v$ are unit representatives of two projective directions and $\theta(u,v)$ is their intrinsic link distance, then
\begin{equation}\label{eq:cone-link-distance}
        d_{\mathcal T_4}(r u,s v)^2
        =r^2+s^2-2rs\cos\bigl(\min\{\theta(u,v),\pi\}\bigr),
        \qquad r,s\ge0.
\end{equation}
Thus, the code computes the intrinsic link distances once for a finite set of projective directions and then optimizes the radial coordinate continuously.  For a fixed projective candidate $u$ and exponent $p$, the one-dimensional problem
$$\min_{r\ge0}\;\frac1k\sum_{j=1}^k d_{\mathcal T_4}(r u,x_j)^p$$
is solved numerically by a bounded one-dimensional search.  This effectively reduces by one the dimension of the distance computation problem.

Projective candidates are generated by successive nested grids.  The first stage is global; later stages build finer local grids around the previously selected $\phi_p$-means, including neighbouring cells and a prescribed minimum number of nearby projective directions.

For the link-distance computation, the code enumerates admissible cell corridors in the triangulated complex. On each corridor, it first attempts an exact unfolding calculation: adjacent Euclidean triangles are unfolded across shared faces, and the resulting straight segment is accepted when its crossing points lie in the required faces. The resulting distance matrices are cached.

The main-text visualization in Figure \ref{fig:ultra-pu4-embedding} shows the projectivized samples and $\phi$-means in a cellwise isometric three-dimensional immersion of $\widetilde{\mathcal T}_4$. This is supplemented by Figure \ref{fig:ultra-means}, which renders selected computed ultrametric centres as dendrograms. These dendrograms are visualizations of the selected points of \(\mathcal T_4\), not their projectivation.

\begin{figure}[!htbp]
  \centering
  \includegraphics[width=\linewidth]{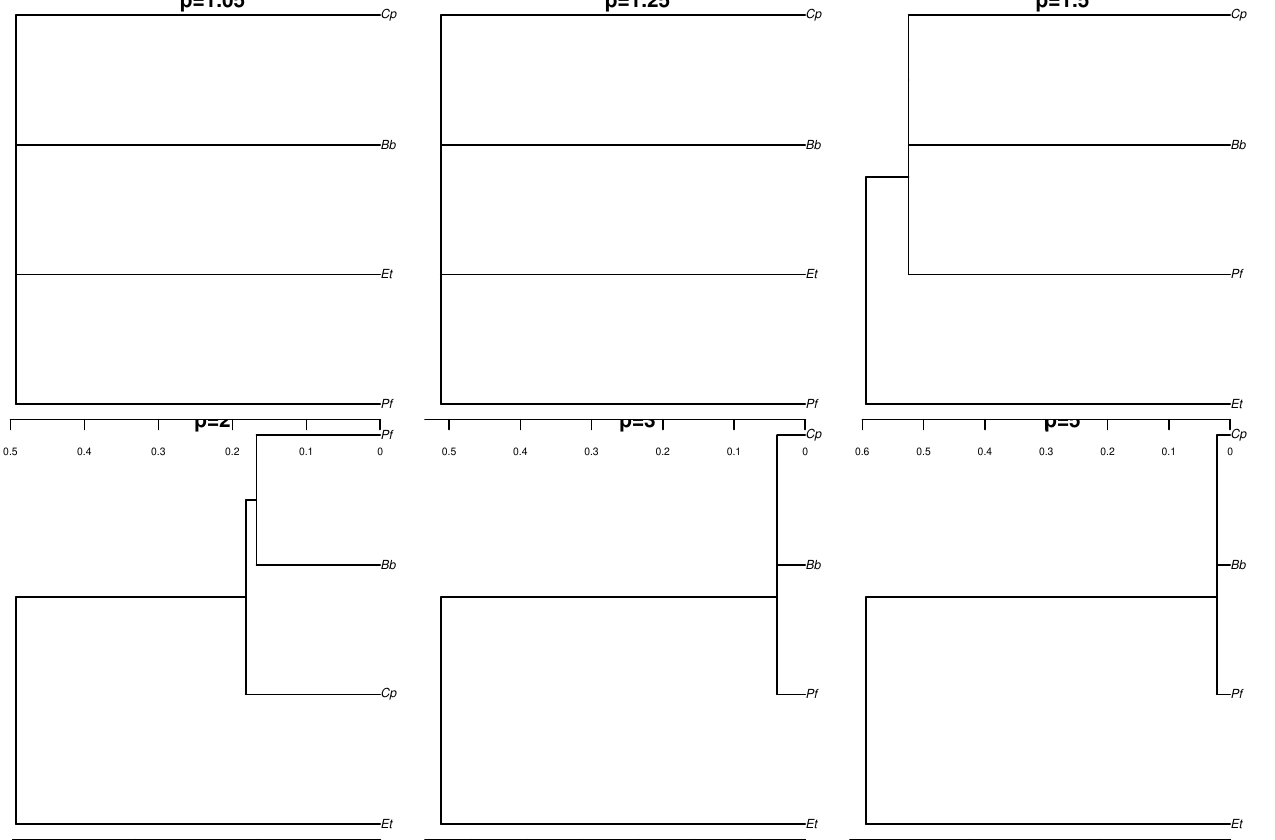}
  \caption{Dendrogram renderings of selected nested-mesh $\phi_p$-means in \(\mathcal T_4\).  The renderings are for interpretation only; the optimization is performed intrinsically in the four-taxon ultrametric cone.}
  \label{fig:ultra-means}
\end{figure}

\section{Isometric Immersions of Tree Spaces}\label{AAUltrametricSymmetries}\label{app:Un-symmetries}

We provide the symmetry statement used in the ultrametric application. Since the projectivized space lies in an affine hyperplane, the natural ambient space for its Euclidean symmetries is that affine span. Put $N=\binom n2$ and
$$\Delta_N=\left\{d\in\R^N:\sum_{i<j}d_{ij}=1\right\}.$$
The group considered below consists of Euclidean affine isometries of $\Delta_N$ that preserve the standard laminar cell complex. Equivalently, these
are the cellular Euclidean symmetries of the projective ultrametric complex with the metric induced from the coordinate embedding.
The symmetric group $S_n$ acts on $\R^N$ by permuting unordered pairs:
$$(\sigma\cdot d)_{ij}=d_{\sigma^{-1}(i)\sigma^{-1}(j)}.$$
This action is orthogonal and preserves both the ultrametric cone
${\mathcal T}_n$ and the projective slice $\widetilde{\mathcal T}_n={\mathcal T}_n\cap \Delta_N.$
For $A\subseteq[n]$ with $|A|\ge2$, define a normalized ray vector $v_A\in \widetilde{\mathcal T}_n$ by
$$(v_A)_{ij}=\left(\binom n2-\binom{|A|}{2}\right)^{-1}\mathbf1_{\{i,j\}\not\subseteq A}$$
for $A\neq[n]$, and set $v_{[n]}=(1/N,\ldots,1/N)$.
The cells of $\widetilde{\mathcal T}_n$ are precisely
\begin{align*}
\sigma_{\mathcal F}
=
\operatorname{conv}\{v_A:A\in\mathcal F\},
\end{align*}
where $\mathcal F$ is a laminar family of non-singleton subsets containing
$[n]$.

\begin{proposition}[Cellular ambient symmetry group of the standard ultrametric embedding]\label{prop:Un-symmetry}
For $n\ge3$, the natural action of $S_n$ induces an isomorphism
\begin{align*}
S_n\cong
\operatorname{Aut}^{\rm cell}_{\rm Euc (\Delta_N)}\bigl(\widetilde{\mathcal T}_n\bigr)=\operatorname{Aut}^{\rm cell}_{\rm Euc (\R^N)}\bigl({\mathcal T}_n\bigr),
\end{align*}
where the group on the right denotes Euclidean affine isometries of $\Delta_N$ that carry $\widetilde{\mathcal T}_n$ to itself and carry each canonical laminar cell to a canonical laminar cell.  In other words, every cellular Euclidean symmetry of the projective ultrametric complex, considered in its affine span, is induced by a unique relabeling of the leaves:
\begin{align*}
T(d)_{ij}=d_{\sigma^{-1}(i)\sigma^{-1}(j)}
\end{align*}
for some $\sigma\in S_n$.
For the standard conical embedding ${\mathcal T}_n\subset\R^N$, the cellular linear orthogonal symmetry group is again the natural copy of $S_n$.
\end{proposition}

\begin{proof}
The inclusion of $S_n$ is immediate from the definition of the ultrametric inequalities, the laminar cells, and the Euclidean inner product.

Conversely, let $T$ be a cellular Euclidean isometry of
$\widetilde{\mathcal T}_n\subset \Delta_N$.  Since $T$ is cellular, it permutes the zero-cells $v_A$.  The vertex $v_{[n]}$ is combinatorially distinguished: it belongs to every maximal simplex, because every maximal laminar family contains $[n]$, whereas no other vertex has this property.
Indeed, if $A\subsetneq[n]$ is a proper non-singleton cluster, choose $a\in A$ and $b\notin A$; a binary laminar family containing the cherry $\{a,b\}$ is incompatible with $A$ and hence avoids $A$.
Therefore
\begin{align*}
T(v_{[n]})=v_{[n]}.
\end{align*}

For a proper cluster $A\subsetneq[n]$ with $|A|=k$, a direct computation gives
\begin{align*}
\|v_A-v_{[n]}\|^2
=
\frac{\binom k2}{N\left(N-\binom k2\right)}.
\end{align*}
This quantity is strictly increasing in $k$ for $2\le k<n$.  Hence, any cellular Euclidean symmetry preserves the cardinality of the proper cluster indexing a zero-cell.  In particular, it preserves the set of $(n-1)$-cluster vertices
\begin{align*}
v_{[n]\setminus\{i\}},
\qquad i=1,\ldots,n.
\end{align*}
Thus $T$ determines a permutation $\sigma\in S_n$ by
\begin{align*}
T\bigl(v_{[n]\setminus\{i\}}\bigr)
=
v_{[n]\setminus\{\sigma(i)\}}.
\end{align*}

We now recover the image of every other vertex from incidence with these $(n-1)$-cluster vertices.  Let $A\subsetneq[n]$ be a non-singleton cluster with $|A|\le n-2$.  Then $A$ is compatible with $[n]\setminus\{i\}$ if and only if $i\notin A$.  Equivalently, $v_A$ is adjacent in the one-skeleton to exactly those vertices $v_{[n]\setminus\{i\}}$ with $i\notin A$.  Since $T$ preserves the canonical cell complex and hence the one-skeleton incidence relation, if $T(v_A)=v_B$, then
\begin{align*}
[n]\setminus B
=
\sigma([n]\setminus A),
\end{align*}
and so $B=\sigma(A)$. Therefore, $T$ agrees with the leaf relabeling
$\sigma$ on every zero-cell.

Finally, the two-cluster vertices affinely span $\Delta_N$.  Indeed, for $|A|=2$, the points $v_A$ are affinely independent and lie in $\Delta_N$.  Hence, an affine isometry of $\Delta_N$ is determined by its values on these vertices. The restriction of $T$ to $\Delta_N$ is therefore exactly the coordinate permutation induced by $\sigma$.  This proves the projective statement.

For the cone ${\mathcal T}_n$, a cellular linear orthogonal symmetry permutes the extreme rays.  The same argument applies to these rays: the ray $\R_{\ge0}v_{[n]}$ is the unique ray contained in all maximal cones, the angles from it distinguish cluster cardinalities, and incidence with the $(n-1)$-cluster rays recovers the labels. Since the two-cluster ray vectors linearly span $\R^N$, the orthogonal symmetry is the coordinate permutation induced by a unique $\sigma\in S_n$.
\end{proof}

\begin{remark}[Exact four-leaf symbolic check]\label{rem:PU4-mathematica-symmetry}
For $n=4$, the preceding conclusion was also checked by a finite exact
calculation performed in Mathematica. In particular, we constructed the eleven rational vertices $v_A\in \Delta_6\subset\mathbb Q^6$, computed their exact squared-distance matrix, and enumerated all distance-matrix automorphisms.
We verified that the finite metric carried by the embedded zero-skeleton has symmetry group $S_4$.  We also verified the automorphism group of the Petersen incidence graph of the proper clusters, which is the group $S_5$ of order 120. Thus, the additional $S_5$-symmetries of the abstract Petersen graph are not symmetries of the standard Euclidean ultrametric embedding.
The exact squared-distance values are given in Table \ref{tab:u4-eleven-vertices}.

\begin{table}
\centering
\small
\begin{minipage}[t]{0.45\textwidth}
\centering
\begin{tabular}{c|c}
$v_A$ & $(d_{12},d_{13},d_{14},d_{23},d_{24},d_{34})$\\\hline
$v_{12}$ & $(0,\nicefrac15,\nicefrac15,\nicefrac15,\nicefrac15,\nicefrac15)$\\
$v_{13}$ & $(\nicefrac15,0,\nicefrac15,\nicefrac15,\nicefrac15,\nicefrac15)$\\
$v_{14}$ & $(\nicefrac15,\nicefrac15,0,\nicefrac15,\nicefrac15,\nicefrac15)$\\
$v_{23}$ & $(\nicefrac15,\nicefrac15,\nicefrac15,0,\nicefrac15,\nicefrac15)$\\
$v_{24}$ & $(\nicefrac15,\nicefrac15,\nicefrac15,\nicefrac15,0,\nicefrac15)$\\
$v_{34}$ & $(\nicefrac15,\nicefrac15,\nicefrac15,\nicefrac15,\nicefrac15,0)$\\
$v_{123}$ & $(0,0,\nicefrac13,0,\nicefrac13,\nicefrac13)$\\
$v_{124}$ & $(0,\nicefrac13,0,\nicefrac13,0,\nicefrac13)$\\
$v_{134}$ & $(\nicefrac13,0,0,\nicefrac13,\nicefrac13,0)$\\
$v_{234}$ & $(\nicefrac13,\nicefrac13,\nicefrac13,0,0,0)$\\
$v_{1234}$ & $(\nicefrac16,\nicefrac16,\nicefrac16,\nicefrac16,\nicefrac16,\nicefrac16)$\\
\end{tabular}
\smallskip

\textbf{(a)} Coordinates of the eleven vertices $v_A$ of $\widetilde{\mathcal T}_4\subseteq \Delta_6$ for all $A\subseteq[4]$ with $|A|>1$.
\end{minipage}
\hfill
\begin{minipage}[t]{0.45\textwidth}
\centering

\begin{tabular}{c|c|c}
\text{pair of vertices} & $\|v_A-v_B\|^2$ & \text{edge?}\\
\hline
$v_{ij},v_{1234}$&$\nicefrac{1}{30}=0.0\bar{3}$& \text{yes}\\
$v_{ij},v_{kl}$&$\nicefrac{2}{25}=0.08$& \text{yes}\\
$v_{ij},v_{ik}$&$\nicefrac{2}{25}=0.08$& \text{no}\\
$v_{ij},v_{ijk}$&$\nicefrac{2}{15}=0.1\bar3$& \text{yes}\\
$v_{ijk},v_{1234}$&$\nicefrac{1}{6}=0.1\bar6$& \text{yes}\\
$v_{ij},v_{ikl}$&$\nicefrac{4}{15}=0.2\bar6$& \text{no}\\
$v_{ijk},v_{ijl}$&$\nicefrac{4}{9}=0.\bar{4}$& \text{no}\\
\end{tabular}
\smallskip
\textbf{(b)} Squared Euclidean chord lengths.  Rows marked "yes" are actual edges of the complex, hence their lengths are fixed by any cellwise isometric immersion; rows marked ``no'' are non-edge chord lengths in this standard embedding.
\end{minipage}
\caption{The eleven standard vertices of $\widetilde{\mathcal T}_4\subseteq\Delta_6$ and the corresponding cases for squared Euclidean chord lengths.}
\label{tab:u4-eleven-vertices}
\end{table}

\end{remark}

\begin{remark}[The exceptional-looking $\widetilde{\mathcal T}_4$ picture]
For $n=4$, the link of the distinguished vertex $v_{1234}$ is the Petersen
graph. The abstract Petersen graph has automorphism group $S_5$, but the
standard Euclidean metric on the zero-skeleton distinguishes the full-cluster
vertex, the six two-cluster vertices, and the four three-cluster vertices, and
then the exact pair--triple distances recover the leaf labels. Consequently,
the displayed Petersen model has only the expected leaf-relabeling symmetries
as an embedded ultrametric object, even though the underlying abstract graph has
a larger automorphism group.
\end{remark}

\section{The \texorpdfstring{$\mathbb Z_3$}{Z3}-Symmetric Four-Leaf Projective Immersions: Classification}\label{app:PU4-Z3-summary}

We now complete the finite algebraic check for the $\mathbb Z_3$-symmetric
realizations of $\widetilde{\mathcal T}_4$ in $\mathbb R^3$.  The vertices of
$\widetilde{\mathcal T}_4$ are indexed by the non-singleton subsets
$A\subseteq\{1,2,3,4\}$.  We write $v_A$ for the corresponding vertex and
write $v_{1234}$ for the fully unresolved vertex.  A top-dimensional cell is
spanned by $v_{1234}$ and two compatible proper clusters.  Hence, a simplicial
map is cellwise isometric if and only if it preserves all one-skeleton edge
lengths, namely the ten edges from $v_{1234}$ to a proper cluster and the
fifteen Petersen edges between compatible proper clusters.

Let $R$ denote rotation by $2\pi/3$ around the $z$-axis,
\begin{align*}
R(x,y,z)=\left(-\frac{x}{2}-\frac{\sqrt3}{2}y,\frac{\sqrt3}{2}x-\frac{y}{2},z\right).
\end{align*}
After quotienting by Euclidean rigid motions and by reflection in a vertical
plane, every normalized $\mathbb Z_3$-equivariant realization can be written in
the following form:
\begin{align*}
I(v_{1234})=0,\qquad
I(v_{123})=Q:=\left(0,0,\frac{\sqrt6}{6}\right),
\end{align*}
\begin{align*}
I(v_{12})=A:=\left(\frac{\sqrt6}{15},0,\frac{\sqrt6}{30}\right),
\qquad
I(v_{14})=B=(x,y,z),
\qquad
I(v_{124})=C=(X,Y,Z).
\end{align*}
The remaining vertices are then forced by $\mathbb Z_3$-equivariance:
$$ I(v_{23})=RA,\ I(v_{13})=R^2A,\ I(v_{24})=RB,\ I(v_{34})=R^2B,\ I(v_{234})=RC,\ I(v_{134})=R^2C. $$
The constants above are forced by the three edge lengths
\begin{align*}
\|v_{123}-v_{1234}\|^2=\nicefrac16,
\qquad
\|v_{12}-v_{1234}\|^2=\nicefrac1{30},
\qquad
\|v_{12}-v_{123}\|^2=\nicefrac2{15}.
\end{align*}
The remaining edge-length constraints are equivalent to the six equations
\begin{align}
\|B\|^2&=\nicefrac1{30},
&
\|C\|^2&=\nicefrac16,\label{eq:PU4Z3-system-a}\\
\langle A,R^2B\rangle&=-\nicefrac1{150},
&
\langle A,C\rangle&=\nicefrac1{30},\label{eq:PU4Z3-system-b}\\
\langle B,C\rangle&=\nicefrac1{30},
&
\langle RB,C\rangle&=\nicefrac1{30}.\label{eq:PU4Z3-system-c}
\end{align}
Indeed, these equations represent respectively the edge lengths
from $v_{14}$ and $v_{124}$ to $v_{1234}$, the balanced edge
$12|34$, the nested edge $12<124$, and the two nested edges
$14<124$ and $24<124$.  Applying $R$ gives all other Petersen-edge constraints.

Set $s_2=\sqrt2$, $s_3=\sqrt3$, and $s_6=\sqrt6$.  From
$\langle A,C\rangle=1/30$ we get $X=(1/s_6-Z)/2.$
Eliminating $X,Y,x,y,z$ from \eqref{eq:PU4Z3-system-a}--\eqref{eq:PU4Z3-system-c} gives the univariate equation
\begin{align*}
Z\left(Z-\frac{s_6}{6}\right)\left(Z^2-\frac1{18}\right)=0.
\quad \textit{Thus,}\quad
Z\in\left\{-\frac{s_2}{6},\,0,\,\frac{s_2}{6},\,\frac{s_6}{6}\right\}.
\end{align*}
For each such $Z$, the value of $X$ is as above and
\begin{align*}
Y^2=\frac16-Z^2-X^2.
\end{align*}
If $Z\ne s_6/6$, the two signs of $Y$ give two solutions and then
\eqref{eq:PU4Z3-system-b}--\eqref{eq:PU4Z3-system-c} determine $B$ uniquely.
If $Z=s_6/6$, then $C=Q$ and the linear equations for $B$ have two solutions. Consequently, there are exactly eight normalized $\mathbb Z_3$-equivariant cellwise isometric realizations.  They are the following:
$$
\begin{array}{c|c|c}
\# & C=(X,Y,Z) & B=(x,y,z)\\
\hline
1 & \left(\dfrac{s_2+s_6}{12},\dfrac{\sqrt{2-s_3}}{6},-\dfrac{s_2}{6}\right)
  & \left(\dfrac{s_2-s_6}{30},\dfrac{s_6-s_2}{30},-\dfrac{2s_2+s_6}{30}\right)\\[3mm]
2 & \left(\dfrac{s_2+s_6}{12},-\dfrac{\sqrt{2-s_3}}{6},-\dfrac{s_2}{6}\right)
  & \left(\dfrac{s_6-2s_2}{30},\dfrac{s_2}{30},-\dfrac{2s_2+s_6}{30}\right)\\[3mm]
3 & \left(\dfrac{s_6}{12},\dfrac{s_2}{4},0\right)
  & \left(\dfrac{s_6}{15},0,\dfrac{s_6}{30}\right)\\[3mm]
4 & \left(\dfrac{s_6}{12},-\dfrac{s_2}{4},0\right)
  & \left(-\dfrac{s_6}{30},-\dfrac{s_2}{10},\dfrac{s_6}{30}\right)\\[3mm]
5 & \left(\dfrac{s_6-s_2}{12},\dfrac{\sqrt{2+s_3}}{6},\dfrac{s_2}{6}\right)
  & \left(\dfrac{s_6+2s_2}{30},\dfrac{s_2}{30},\dfrac{2s_2-s_6}{30}\right)\\[3mm]
6 & \left(\dfrac{s_6-s_2}{12},-\dfrac{\sqrt{2+s_3}}{6},\dfrac{s_2}{6}\right)
  & \left(-\dfrac{s_6+s_2}{30},-\dfrac{s_6+s_2}{30},\dfrac{2s_2-s_6}{30}\right)\\[3mm]
7 & \left(0,0,\dfrac{s_6}{6}\right)
  & \left(-\dfrac{s_6}{30},-\dfrac{s_2}{10},\dfrac{s_6}{30}\right)\\[3mm]
8 & \left(0,0,\dfrac{s_6}{6}\right)
  & \left(\dfrac{s_6}{15},0,\dfrac{s_6}{30}\right).
\end{array}
$$
Together with the $R$-orbits described above, this table gives every normalized $\mathbb Z_3$-symmetric cellwise isometric realization of $\widetilde{\mathcal T}_4$ in $\mathbb R^3$.

The first, second, fifth, and sixth solutions separate all eleven vertices of
$\widetilde{\mathcal T}_4$.  The third and fourth identify the orbit
$\{12,23,13\}$ with the orbit $\{14,24,34\}$ in one of the two possible cyclic
alignments. The seventh and eighth additionally collapse the orbit
$\{124,234,134\}$ to the fixed point $v_{123}$ on the symmetry axis. Thus, if
one requires the realization to be injective on the zero-skeleton, exactly four
normalized $\mathbb Z_3$-symmetric solutions remain, namely rows $1,2,5,6$.

The classification above concerns the projective slice $\widetilde{\mathcal T}_4$.
It remains to check whether one can choose a cone-apex vector so that the projective realization extends to an isometric realization of the full cone ${\mathcal T}_4$ with its ambient Euclidean metric.  The answer is negative in $\mathbb R^3$ and positive after adding one orthogonal coordinate; this is the content of the next section.

\begin{figure}[!htbp]
  \centering
  \includegraphics[width=0.98\linewidth]{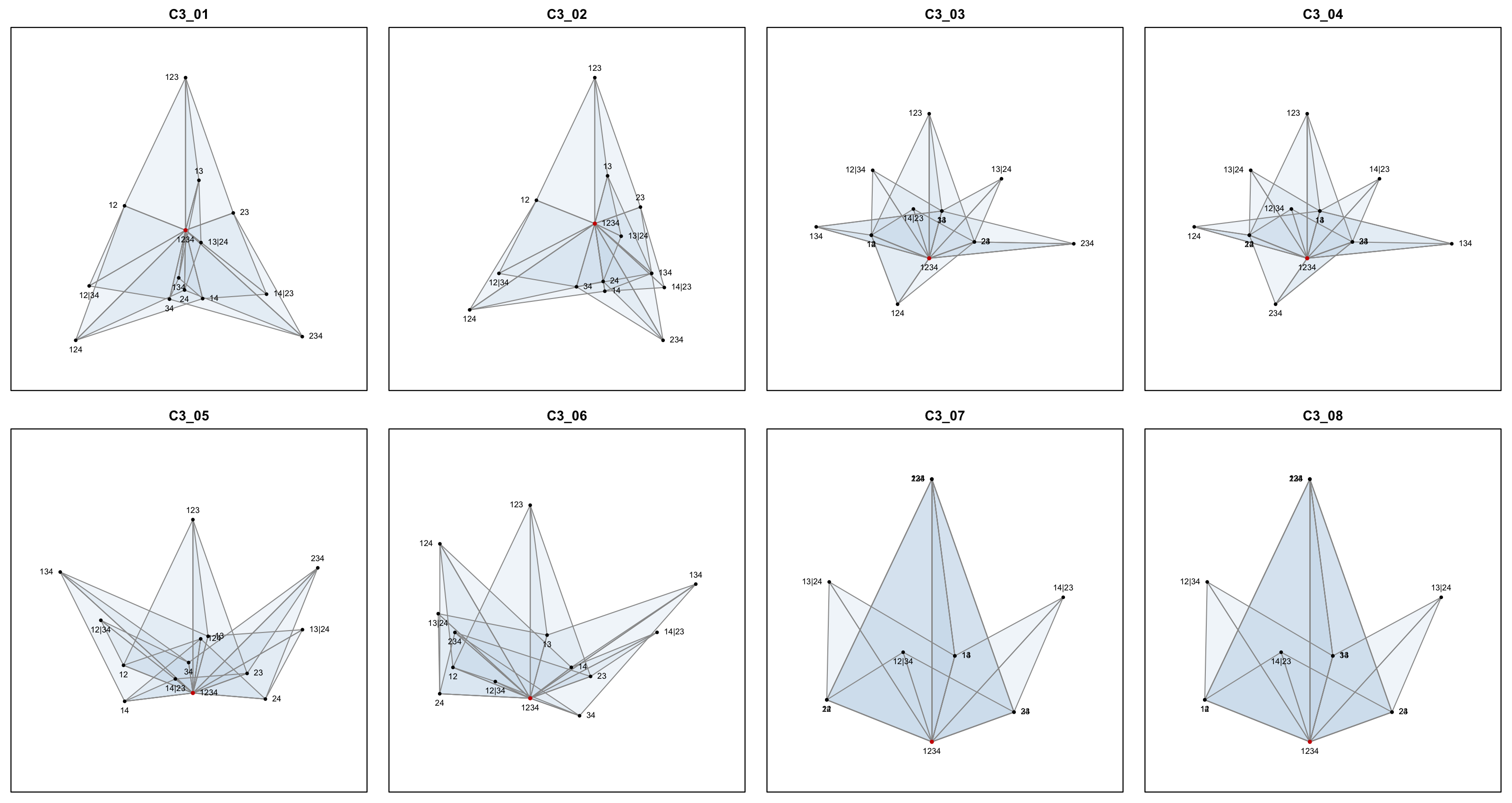}
  \caption{All eight $\Z_3$-symmetric isometric immersions of $\widetilde{\mathcal T}_4$.}
  \label{fig:wowimmersions}
\end{figure}

\section{Isometric Immersions of \texorpdfstring{${\mathcal T}_4$}{T4}}\label{app:pu4-cone-apex-gram}

The preceding section classifies the normalized $\mathbb Z_3$-symmetric cellwise isometric realizations of the projectivized space $\widetilde{\mathcal T}_4$ in $\mathbb R^3$.  We now check whether any of these projective realizations can be radially extended to an isometric realization of the full cone ${\mathcal T}_4\subset\mathbb R^6$, equipped with the ambient Euclidean metric.

Let $f=v_{1234}$ denote the full-cluster vertex of $\widetilde{\mathcal T}_4$. In the normalization used above, we have
$$I(f)=0,
\qquad
I(v_{123})=Q=\left(0,0,\frac{\sqrt6}{6}\right),
\qquad
I(v_{12})=A=\left(\frac{\sqrt6}{15},0,\frac{\sqrt6}{30}\right),$$
and by equivariance
$$I(v_{23})=RA,
\qquad
I(v_{13})=R^2A,$$
where $R$ is rotation by $2\pi/3$ around the $z$-axis.  These coordinates are common to all eight normalized solutions.

Suppose that a cone-apex vector $c\in\mathbb R^3$ existed such that $J(v_A):=I(v_A)+c$
had the same cone Gram data as the original ray vectors $v_A\in{\mathcal T}_4\subset\mathbb R^6$.  Since every $v_A$ lies in the affine slice $\sum_{i<j}d_{ij}=1$, the full-cluster vertex satisfies
\begin{align*}
\|v_{1234}\|^2=\nicefrac{1}{6}.
\end{align*}
Moreover
$$\|v_A\|^2=
\begin{cases}
1/5, & |A|=2,\\[1mm]
1/3, & |A|=3,\\[1mm]
1/6, & A=1234.
\end{cases}$$
Thus, the necessary cone-apex norm conditions are
$$\|c\|^2 = \nicefrac{1}{6},\ \|Q+c\|^2 = \nicefrac{1}{3},\ \|A+c\|^2 = \nicefrac{1}{5},\ \|RA+c\|^2 = \nicefrac{1}{5},\ \|R^2A+c\|^2 = \nicefrac{1}{5}.$$
By simple algebra and by the fact that $\|Q\|^2=1/6$, $\|A\|^2=1/30$ these conditions imply $\langle Q,c\rangle=\langle A,c\rangle=\langle RA,c\rangle=\langle R^2A,c\rangle=0$
But $Q, A, RA$ span $\mathbb R^3$. Hence $c=0$, contradicting $\|c\|^2=1/6$. Therefore, no such cone-apex vector exists. The supplementary symbolic check records this last step in the equivalent exact form $\operatorname{rank}\{Q, A, RA\}=3$ and confirms that the corresponding real Gram system has no solution.
Consequently, none of the eight normalized $\mathbb Z_3$-symmetric isometric realizations of $\widetilde{\mathcal T}_4$ in $\mathbb R^3$ admits a radial lift to an isometric realization of the full Euclidean cone ${\mathcal T}_4$ in the same ambient space $\mathbb R^3$.

There is, however, a canonical lift in one additional Euclidean dimension.  Indeed, since $I$ realizes the translated projective slice $v_A-v_{1234}$, define
$$\widehat I(v_A)=\left(I(v_A),\frac{1}{\sqrt6}\right)\in\mathbb R^4.$$
Then, for all vertices $v_A,v_B$ belonging to a common cone cell,
$$\langle \widehat I(v_A),\widehat I(v_B)\rangle=\langle I(v_A),I(v_B)\rangle+\frac{1}{6}=\langle v_A,v_B\rangle_{\mathbb R^6}.$$
By linear extension on each cone cell, this gives a $\mathbb Z_3$-equivariant isometric immersion of ${\mathcal T}_4$ in $\mathbb R^4$.

\end{document}